\documentclass[secthm,seceqn,amsthm,ussrhead,12pt]{amsart}
\usepackage[utf8]{inputenc}
\usepackage[english]{babel}

\usepackage{amssymb,amsmath,amsthm,amsfonts,xcolor,enumerate,hyperref,comment,longtable,cleveref}
\usepackage{soul}
\usepackage{times}
\usepackage{cite}
\usepackage{pdflscape}
\usepackage{ulem}
\usepackage[mathcal]{euscript}
\usepackage{tikz}
\usepackage{hyperref}
\usepackage{cancel}
\usepackage{stmaryrd}

\usetikzlibrary{arrows}

\newcommand{\Co}{{\mathbb C}}

\newtheorem{theorem}{Theorem}
\newtheorem{lemma}[theorem]{Lemma}
\newtheorem{proposition}[theorem]{Proposition}
\newtheorem{remark}[theorem]{Remark}
\newtheorem{example}[theorem]{Example}
\newtheorem{corollary}[theorem]{Corollary}

\theoremstyle{definition}
\newtheorem{definition}[theorem]{Definition}

\newtheorem*{theoremA}{Theorem A}
\newtheorem*{theoremB}{Theorem B}
\newtheorem*{theoremC}{Theorem C}

\newtheorem*{question}{Open question}
\newtheorem*{conjecture}{Conjecture}

\newenvironment{Proof}[1][Proof.]{\begin{trivlist}
\item[\hskip \labelsep {\bfseries #1}]}{\flushright
$\Box$\end{trivlist}}

\usepackage{stmaryrd}
\usepackage{xcolor}

\newcommand{\aut}[1]{\operatorname{\mathrm{Aut}}{#1}}

\newcommand{\cA}{\mathcal{A}}
\newcommand{\cB}{\mathcal{B}}

\newcommand{\R}{\mathcal{R}}
\newcommand{\cS}{\mathcal{S}}
\newcommand{\cT}{\mathcal{T}}

\newcommand{\V}{\mathfrak{V}}
\newcommand{\nil}{\mathfrak{Nil}}

\newcommand{\la}{\langle}
\newcommand{\ra}{\rangle}
\newcommand{\La}{\Big\langle}
\newcommand{\Ra}{\Big\rangle}

\newcommand{\Dt}[2]{\Delta_{#1#2}}
\newcommand{\Dl}[2]{[\Delta_{#1#2}]}

\newcommand{\0}{\theta}
\newcommand{\af}{\alpha}

\newcommand{\gm}{\gamma}

\newcommand{\vf}{\varphi}

\newcommand{\ann}[1]{\operatorname{\mathrm{Ann}}{#1}}

\begin{document}

{\Large\noindent 
The geometric classification of nilpotent algebras
\footnote{
The authors thank  Alexander Guterman for some constructive discussions about the length of algebras.
The work was partially  supported by  
CNPq 404649/2018-1, 302980/2019-9; RFBR 20-01-00030 and by CMUP, which is financed by national funds through FCT---Funda\c c\~ao para a Ci\^encia e a Tecnologia, I.P., under the project with reference UIDB/00144/2020.
 }
}

\ 

   {\bf Ivan Kaygorodov$^{a},$ Mykola Khrypchenko$^{b}$ \&   Samuel A.\ Lopes$^{c}$  \\

    \medskip
}

{\tiny

$^{a}$ CMCC, Universidade Federal do ABC, Santo Andr\'e, Brazil

$^{b}$ Departamento de Matem\'atica, Universidade Federal de Santa Catarina, Florian\'opolis, Brazil
 
$^{c}$ CMUP, Faculdade de Ci\^encias, Universidade do Porto, Rua do Campo Alegre 687, 4169-007 Porto, Portugal

\

\smallskip

   E-mail addresses:

\smallskip
    Ivan Kaygorodov (kaygorodov.ivan@gmail.com)

\smallskip  
    Mykola Khrypchenko (nskhripchenko@gmail.com)    

\smallskip    
    Samuel A.\ Lopes (slopes@fc.up.pt)

}

\

\ 

\

\noindent{\bf Abstract}: 
{\it 
We give a geometric classification of $n$-dimensional nilpotent, commutative nilpotent and anticommutative nilpotent algebras. We prove that the corresponding geometric varieties are irreducible, find their dimensions and describe explicit generic families of algebras which define each of these varieties. We show some applications of these results in the study of the length of anticommutative algebras.}

\


\ 


\section*{Introduction}

The geometry of varieties of algebras defined by polynomial identities has been an active area of interest and research since the work of Nijenhuis--Richardson~\cite{NR66} and Gabriel~\cite{gabriel} in the 1960's and 1970's. The relationship between geometric features of the variety (such as irreducibility, dimension, smoothness) and the algebraic properties of its points brings novel geometric insight into the structure of the variety, its generic points and degenerations. 

Given algebras ${\mathcal A}$ and ${\mathcal B}$ in the same variety, we write ${\mathcal A}\to {\mathcal B}$ and say that ${\mathcal A}$ {\it degenerates} to ${\mathcal B}$, or that ${\mathcal A}$ is a {\it deformation} of ${\mathcal B}$, if ${\mathcal B}$ is in the Zariski closure of the orbit of ${\mathcal A}$ (under the base-change action of the general linear group). The study of degenerations of algebras is very rich and closely related to deformation theory, in the sense of Gerstenhaber \cite{ger63}. Degenerations have also been used to study a level of complexity of an algebra~\cite{gorb93,kv17}. There are many results concerning degenerations of algebras of small dimensions in a variety defined by a set of identities (see, for example, \cite{kv17, kv19, kpv19, kkp19geo, kv19, kkl19, ikv17, ikp20, GRH, GAB92, fkkv, BC99, ale3} and references therein). An interesting question is to study those properties which are preserved under degenerations. Recently, Chouhy~\cite{chouhy} proved that in the case of finite-dimensional associative algebras, the $N$-Koszul property is one such property.

Concerning Lie algebras, Grunewald--O'Halloran~\cite{GRH} calculated the degenerations for the variety of $5$-dimensional nilpotent Lie algebras while in~\cite{BC99}, Burde and Steinhoff constructed the graphs of degenerations for the varieties of $3$- and $4$-dimensional Lie algebras and in~\cite{fkkv} Fern\'andez Ouaridi, Kaygorodov, Khrypchenko and Volkov described the full graphs of degenerations of small dimensional nilpotent algebras.

One of the main problems of the {\it geometric classification} of a variety of algebras is a description of its irreducible components. In~\cite{gabriel}, Gabriel described the irreducible components of the variety of $4$-dimensional unital associative algebras and the variety of $5$-dimensional unital associative algebras was classified algebraically and geometrically by Mazzola~\cite{maz79}. Later, Cibils~\cite{cibils} considered rigid associative algebras with $2$-step nilpotent radical. Goze and Ancoch\'{e}a-Berm\'{u}dez proved that the varieties of $7$- and  $8$-dimensional nilpotent Lie algebras are reducible~\cite{GAB92}. All irreducible components of $2$-step nilpotent, commutative nilpotent and anticommutative nilpotent algebras have been described in \cite{shaf,ikp20}.

In many cases, the irreducible components of the variety are determined by the rigid algebras, i.e.\ algebras whose orbit closure is an irreducible component. It is worth mentioning that this is not always the case and Flanigan had shown that the variety of $3$-dimensional nilpotent associative algebras has an irreducible component which does not contain any rigid algebras---it is instead defined by the closure of a union of a one-parameter family of algebras \cite{fF68}. Here, we will encounter similar situations. 
Our main results are based on Theorems \ref{T:Rn:generic}, 
\ref{Sn:generic}, 
\ref{Tn:generic} 
and \cite{fkkv,kv19}. We are summarizing them below.

\begin{theoremA}
For any $n\ge 2$, the variety of all $n$-dimensional nilpotent algebras is irreducible and has dimension $\frac{n(n-1)(n+1)}{3}$.
\end{theoremA}

Moreover, we show that the family $\R_n$ for  $n\ge 3$ given in Definition~\ref{R_n-defn} is generic in the variety of $n$-dimensional nilpotent algebras and inductively give an algorithmic procedure to obtain any $n$-dimensional nilpotent algebra as a degeneration from $\R_n$. 
The case of $n=2$ follows from \cite{kv19}.

\begin{theoremB}
For any $n\ge 2$, the variety of all $n$-dimensional commutative nilpotent algebras is irreducible and has dimension $\frac{n(n-1)(n+4)}{6}$.
\end{theoremB}

As above, we show that the family $\cS_n$ for $n\ge 4$ given in Definition~\ref{S_n-defn} is generic in the variety of $n$-dimensional commutative nilpotent algebras and inductively give an algorithmic procedure to obtain any $n$-dimensional nilpotent commutative algebra as a degeneration from $\cS_n$. 
The cases of $n=2$ and $n=3$ follow from \cite{kv19,fkkv}.

\begin{theoremC}
For any $n\ge 2$, the variety of all $n$-dimensional anticommutative nilpotent algebras is irreducible and has dimension $\frac{(n-2)(n^2+2n+3)}{6}$.
\end{theoremC}

We show also that the family $\cT_n$ for $n\ge 6$ given in Definition~\ref{T_n-defn} is generic in the variety of $n$-dimensional anticommutative nilpotent algebras and inductively give an algorithmic procedure to obtain any $n$-dimensional nilpotent anticommutative algebra as a degeneration from $\cT_n$. 
The cases of $n=2,3,4,5$ follow from \cite{kv19,fkkv}.

The notion of length for nonassociative algebras has been recently introduced in~\cite{guterman}, generalizing the corresponding notion for associative algebras. Using the above result, we show in Section~\ref{S:final} (cf.\ Corollary~\ref{C:final:length}) that the length of an arbitrary (i.e.\ not necessarily nilpotent) $n$-dimensional anticommutative algebra is bounded above by the $n^{\text{th}}$ Fibonacci number, and prove that our bound is sharp.

\section{Varieties of algebras, central extensions and nilpotent algebras}

Throughout this paper, we work over the field $\mathbb C$ of complex numbers and, unless otherwise noted, all vector spaces, linear maps and tensor products will be taken over $\mathbb C$. The identity matrix is denoted by $I$ and the matrix unit corresponding to the row $i$ and the column $j$ is $E_{ij}$. For a subset $X$ of a given vector space, the linear span of $X$ is denoted by $\la X\ra$. 

\subsection{Central extensions and the method of Skjelbred and Sund}\label{SS:ceSS}

An algebra is a vector space endowed with a bilinear multiplication. Formally, it is a pair $\cA=({\bf V}, \mu)$, where ${\bf V}$ is a vector space and $\mu\in{\rm Hom}({\bf V} \otimes {\bf V},{\bf V})$, which gives the algebra law. The annihilator of $\cA$ is $\ann\cA=\{ x\in\cA\mid x\cA+\cA x=0 \}$. For the purposes of this paper, we just need to consider $1$-dimensional central extensions. We will give here an overview of the algebraic classification method of Skjelbred and Sund~\cite{SS:method}. 

A bilinear map $\theta\colon {\cA} \times {\cA} \longrightarrow \Co$ determines the $1$-dimensional central extension ${\cA}_{\theta}={\cA}\oplus \Co$ with the product $(x+v)\cdot_{\theta}(y+w)=xy+\theta(x,y)$, for all $x, y\in\cA$ and $v, w\in\Co$. We let ${\rm Z}^{2}\left( {\cA},\Co\right)$ be the vector space of all bilinear maps ${\cA} \times {\cA} \longrightarrow \Co$, which we refer to as $2$-cocycles (of $\cA$ with values in $\Co$). Then, the space of $2$-coboundaries is ${\rm B^2}(\cA,\Co)=\{ \delta f\mid f\in \cA^*\}\subseteq {\rm Z}^2\left( {\cA},\Co\right)$, where $\delta f=f\circ\mu$, so that $\delta f(x,y)=f(xy)$, for all $x, y\in\cA$. The second cohomology of $\cA$ with values in $\Co$ is the quotient space ${\rm H}^2(\cA,\Co) = {\rm Z}^2\left( {\cA},\Co\right) \big/{\rm B^2}(\cA,\Co)$, which is well known to parametrize equivalence classes of $1$-dimensional central extensions of $\cA$.

Once a basis $(e_i)_{i=1}^n$ of $\cA$ is fixed, define the bilinear maps $\Delta_{ij}=e_i^*\times e_j^*$, for $1\leq i, j\leq n$, so that $\Delta_{ij}(e_k,e_\ell)=\delta_{ik}\delta_{j\ell}$, for all $1\leq i, j, k, \ell\leq n$. Then $\left(\Delta_{ij}\right)_{1\leq i, j\leq n}$ is a basis of ${\rm Z}^2(\cA,\Co)$. The automorphism group $\aut{\cA}$ acts on ${\rm Z}^2\left( {\cA},\Co\right)$ via $\phi(\theta)=\theta\circ(\phi\times\phi)$, for $\phi\in\aut{\cA}$ and $\theta\in{\rm Z}^2\left( {\cA},\Co\right)$ and the action induces a well-defined one on ${\rm H}^2(\cA,\Co)$. If $A$ is the matrix of $\theta$ and $M$ is the matrix of $\phi$, then $\phi(\theta)$ has matrix $M^T AM$.

It is easy to see that $\ann{\cA_{\theta}}=\left(\ann\cA\cap\ann{\theta}\right)\oplus\Co$, where $\ann{\theta}=\{x\in{\cA}\mid \theta(x,{\cA})+\theta ({\cA}, x)=0\}$. Thus, given algebras $\cA$, $\cA'$ and respective $2$-cocycles $\theta$, $\theta'$ such that $\ann\cA\cap\ann{\theta}=0=\ann\cA'\cap\ann{\theta'}$, then $\cA_{\theta}\simeq \cA'_{\theta'}$ implies that $\cA\simeq\cA_{\theta}/\ann{\cA_{\theta}}\simeq \cA'_{\theta'}/\ann{\cA'_{\theta'}}\simeq \cA'$. Therefore, it remains to determine the precise conditions on $\theta,\theta'\in{\rm Z}^2\left( {\cA},\Co\right)$ for $\cA_{\theta}$ and $\cA_{\theta'}$ to be isomorphic,
under the assumption that $\ann\cA\cap\ann{\theta}=0=\ann\cA\cap\ann{\theta'}$. This is given by the following result.

\begin{lemma}
Let $\cA$ be an algebra and $\theta,\theta'$ be $2$-cocycles, represented by the nonzero cohomology classes $[\theta]$, $[\theta']$ in ${\rm H}^2(\cA,\Co)$. Suppose that $\ann\cA\cap\ann{\theta}=0=\ann\cA\cap\ann{\theta'}$. Then $\cA_{\theta}$ is isomorphic to $\cA_{\theta'}$ if and only if the orbits of $[\theta]$ and $[\theta']$ under the action of $\aut{\cA}$ span the same vector space.
\end{lemma}

\subsection{Varieties of algebras}

Given an $n$-dimensional vector space ${\bf V}$, an algebra structure on ${\bf V}$ (or an $n$-dimensional algebra law) is naturally seen as an element of ${\rm Hom}({\bf V} \otimes {\bf V},{\bf V}) \cong {\bf V}^* \otimes {\bf V}^* \otimes {\bf V}$, a vector space of dimension $n^3$. Once we fix a basis $e_1,\dots,e_n$ of ${\bf V}$, this space can be identified with $\mathbb{C}^{n^3}$ and given the structure of an affine variety whose coordinate ring if the polynomial ring in the variables $\left(c_{i,j}^k\right)_{i,j,k=1}^n$. Accordingly, a subset of ${\rm Hom}({\bf V} \otimes {\bf V},{\bf V})$ is {\it Zariski-closed} if it can be defined by a set of polynomial equations in the variables $\left(c_{i,j}^k\right)_{i,j,k}$. For simplicity, we identify points in the variety with the corresponding maximal ideals of the coordinate ring so, if no confusion arises, we also think of the $c_{i,j}^k$ as scalars.

Henceforth, having fixed the vector space ${\bf V}$ and its basis $(e_i)_{i=1}^n$, we will identify points $\left(c_{i,j}^k\right)_{i,j,k=1}^n$ in $\mathbb{C}^{n^3}$ with $n$-dimensional algebras via their structure constants relative to the basis $(e_i)_{i=1}^n$. Concretely, any $\mu\in {\rm Hom}({\bf V} \otimes {\bf V},{\bf V})$ is determined by the $n^3$ structure constants $c_{i,j}^k\in\mathbb{C}$ such that
$\mu(e_i\otimes e_j)=\sum_{k=1}^n c_{i,j}^k e_k$.

\begin{example}
The following polynomial identities define well-known varieties of $n$-dimensional algebras: 
\begin{enumerate}
\item Anticommutative algebras
\begin{equation*}
c_{ij}^k+c_{ji}^k=0, \quad 1\leq i, j, k\leq n.
\end{equation*}
\item Associative algebras
\begin{equation*}
\sum_{\ell=1}^n c_{ij}^\ell c_{\ell k}^m- c_{i\ell}^m c_{j k}^\ell=0, \quad 1\leq i, j, k, m\leq n.
\end{equation*}

\item Lie algebras
\begin{align*}
\sum_{\ell=1}^n c_{ij}^\ell c_{k\ell}^m+ c_{jk}^\ell c_{i\ell}^m+c_{ki}^\ell c_{j \ell}^m=0 \quad\mbox{and}\quad c_{ij}^k+c_{ji}^k=0, \quad 1\leq i, j, k, m\leq n.
\end{align*}
\end{enumerate}
\end{example}

\subsection{Nilpotent algebras}\label{SS:nilpalgs}

Recall that an algebra $\mathcal N$ is {\it nilpotent} if ${\mathcal N}^{k}=0$, for some $k\geq 1$, where
\begin{equation*}
{\mathcal N}^{1}={\mathcal N}, \qquad  {\mathcal N}^{k+1}=\sum_{i=1}^{k}{\mathcal N}^i {\mathcal N}^{k+1-i}, \qquad k\geq 1.
\end{equation*}
The smallest (if any) positive integer $k$ satisfying ${\mathcal N}^{k}=0$ is called the {\it nilpotency index} of ${\mathcal N}$. It is easy to see that the $n$-dimensional algebras $\mathcal N$ with ${\mathcal N}^{k}=0$ form a closed set (as elements of $\mathbb{C}^{n^3}$) and thence so do all nilpotent $n$-dimensional algebras.

\begin{lemma}\label{Nil^gm}
Let $\mathcal{N}$ be an algebra of dimension $n$. Then $\mathcal{N}$ is nilpotent if and only if it is isomorphic to an algebra $\mathcal{M}$ whose structure constants $(\gamma_{i,j}^{k})_{i,j,k=1}^n$ satisfy 
\begin{align}\label{g_ij=0-for-k<=max(ij)}
\gamma_{i,j}^{k}=0,\ \forall k \leq \max\{i,j\}.    
\end{align}
\end{lemma}
\begin{proof}
Indeed, \eqref{g_ij=0-for-k<=max(ij)} holds for the algebra with zero multiplication. Moreover, if an algebra $\mathcal{A}$ is a central extension of another algebra $\mathcal{B}$ by some vector space and the structure constants of $\mathcal{B}$ relative to some basis satisfy \eqref{g_ij=0-for-k<=max(ij)}, then completing this basis of $\mathcal{B}$ to a basis of $\mathcal{A}$ we see that the corresponding structure constants of $\mathcal{A}$ will also satisfy \eqref{g_ij=0-for-k<=max(ij)}. It remains to note that each finite-dimensional nilpotent algebra can be obtained, up to isomorphism, from an algebra with zero multiplication by consecutively applying the central extension procedure.

Conversely, if the structure constants of an algebra $\mathcal{M}$ satisfy \eqref{g_ij=0-for-k<=max(ij)}, then all the products in $\mathcal{M}$ of length $2^n$ are zero, so $\mathcal{M}$ is nilpotent.
\end{proof}

\begin{definition}
Let $n\geq 1$. We denote by $\mathfrak{Nil}_n$ the variety of all nilpotent algebra structures (on $\bf V$, with respect to the basis $(e_i)_{i=1}^n$) and by $\mathfrak{Nil}^{\gamma}_n$ its subvariety determined by the system of equations \eqref{g_ij=0-for-k<=max(ij)}.
\end{definition}

Recall that the general linear group $\mathrm{GL}_n(\mathbb{C})$ acts on the space of algebra structures on $\bf V$ via base-change and the orbits parametrize the isomorphism classes of algebras. Given an algebra $\mathcal{A}$, its orbit under this action will be denoted by $O(\mathcal{A})$. Lemma~\ref{Nil^gm} can thus be restated simply as $\mathfrak{Nil}_n=\mathrm{GL}_n(\mathbb{C})\cdot \mathfrak{Nil}^{\gamma}_n$.

\begin{proposition}\label{nilp}
Let $n\geq 1$. The variety $\mathfrak{Nil}_n$ of all nilpotent $n$-dimensional algebras and its subvarieties of commutative and anticommutative nilpotent algebras are irreducible. 
\end{proposition}

\begin{proof}
The coordinate ring $\mathcal{O}_n$ of the variety $\mathfrak{Nil}^{ \gamma}_n$ is the complex polynomial ring in the variables $\left\{c_{i,j}^k\mid 1\leq i,j\leq n, \ k>\max\{i,j\}\right\}$. Thus, $\mathcal{O}_n$ being a domain, the variety $\mathfrak{Nil}^{ \gamma}_n$ is irreducible. If $\ell_n$ is the transcendence degree of $\mathcal{O}_n$, then $\ell_{n+1}=\ell_n+n^2$, so the dimension of $\mathfrak{Nil}^{\gamma}_n$ is $\ell_n=\frac{n(n-1)(2n-1)}{6}$.

By Lemma~\ref{Nil^gm}, $\mathfrak{Nil}_n=\mathrm{GL}_n(\mathbb{C})\cdot \mathfrak{Nil}^{\gamma}_n$. Since $\mathrm{GL}_n(\mathbb{C})$ is a connected (i.e.\ irreducible) algebraic group and the product of irreducible varieties is irreducible, it follows that $\mathrm{GL}_n(\mathbb{C})\times \mathfrak{Nil}^{\gamma}_n$ is irreducible. Hence, so is the continuous image $\mathfrak{Nil}_n=\mathrm{GL}_n(\mathbb{C})\cdot \mathfrak{Nil}^{\gamma}_n$.

A similar argument holds for the varieties of $n$-dimensional commutative and anticommutative nilpotent algebras.
\end{proof}

\begin{definition}
    Let $\V$ be an irreducible variety of algebras and $\R\subseteq\V$ be a family of algebras. The family $\R$ is said to be {\it generic} in $\V$, if $\overline{\bigcup_{\mathcal{A}\in\R} O(\mathcal{A})}=\V$. For an algebra $\cB\in\V$, we also write $\R\to\cB$ as shorthand for $\cB\in \overline{\bigcup_{\mathcal{A}\in\R} O(\mathcal{A})}$.
\end{definition}

\section{The geometric classification of nilpotent algebras}\label{SS:nilpotent algebras}

In this section we find a generic family of algebras in the variety $\mathfrak{Nil}_n$ of all nilpotent $n$-dimensional algebras and use it to compute the dimension of the variety. Recall that $\mathfrak{Nil}^{\gamma}_n$ is the subvariety of algebras satisfying~\eqref{g_ij=0-for-k<=max(ij)} and that $(e_i)_{i=1}^{n}$ is our fixed basis.

\begin{lemma}\label{I+xE_1_n-auto}
    Let $\cA\in\nil^{\gamma}_{n}$. Then, for all $x\in\Co$, the linear endomorphism $\vf$ defined by $\vf(e_1)=e_1+xe_n$, $\vf(e_i)=e_i$, $2\le i\le n$, is an automorphism of $\cA$.
\end{lemma}
\begin{proof}
    Obviously, $\vf(e_i)\vf(e_j)=e_ie_j=\vf(e_ie_j)$ for all $2\le i,j\le n$, in view of~\eqref{g_ij=0-for-k<=max(ij)}. Observe that $e_n\in \ann\cA$, so $\vf(e_1)\vf(e_i)=(e_1+xe_{n})e_i=e_1e_i=\vf(e_1e_i)$ and similarly $\vf(e_i)\vf(e_1)=\vf(e_ie_1)$, for all $2\le i\le n$, by~\eqref{g_ij=0-for-k<=max(ij)}. Finally, $\vf(e_1)^2=(e_1+xe_{n})^2=e_1^2=\vf(e_1^2)$, again by \eqref{g_ij=0-for-k<=max(ij)}. Thus, $\vf$ is an endomorphism of $\cA$. It is clearly invertible, with $\vf^{-1}(e_1)=e_1-xe_{n}$ and $\vf^{-1}(e_i)=e_i$, for all $2\le i\le n$.
\end{proof}

\begin{lemma}\label{Aut(A)-exactly-I+xE_n1}
    Let $\cA\in\nil^{\gamma}_{n+1}$ such that $\ann\cA=\la e_{n+1}\ra$, $e_1e_n=0$ and $e_i^2=e_{i+1}$, for all $1\le i\le n$. Suppose that $\aut(\cA/\la e_{n+1}\ra)=\{I+xE_{n1}\mid x\in\Co\}$ in the basis $(e_i+\la e_{n+1}\ra)_{i=1}^n$. Then $\aut\cA=\{I+xE_{n+1,1}\mid x\in\Co\}$ in the basis $(e_i)_{i=1}^{n+1}$.
\end{lemma}
\begin{proof}
    Let $\vf\in\aut\cA$. Then $\vf(\ann\cA)=\ann\cA$, so $\vf$ induces $\tilde\vf\in\aut(\cA/\la e_{n+1}\ra)$ defined by $\tilde\vf(e_i+\la e_{n+1}\ra)=\vf(e_i)+\la e_{n+1}\ra$, $1\le i\le n$. We know that the matrix of $\tilde\vf$ in $(e_i+\la e_{n+1}\ra)_{i=1}^n$ is of the form $I+xE_{n1}$ for some $x\in\Co$. Since $\vf(\ann\cA)=\ann\cA$, we have $\vf(e_{n+1})=ye_{n+1}$ for some $y\in\Co$. Then, the matrix of $\vf$ in $(e_i)_{i=1}^{n+1}$ is of the form 
    \begin{align*}
        I+xE_{n1}+(y-1)E_{n+1,n+1}+\sum_{j=1}^na_{n+1,j}E_{n+1,j}
    \end{align*}
    for some $\{a_{n+1,j}\}_{j=1}^n\subseteq\Co$. In particular, $\vf(e_n)=e_n+a_{n+1,n}e_{n+1}$. Since $e_n^2=e_{n+1}$, we have
    \begin{align*}
        e_{n+1}=(e_n+a_{n+1,n}e_{n+1})^2=\vf(e_n)^2=\vf(e_n^2)=\vf(e_{n+1})=ye_{n+1}.
    \end{align*}
    Hence, $y=1$. Now, $\vf(e_1)=e_1+xe_n+a_{n+1,1}e_{n+1}$ and $e_1e_n=0$ imply
    \begin{align*}
        0=\vf(e_1e_n)=\vf(e_1)\vf(e_n)=(e_1+xe_n+a_{n+1,1}e_{n+1})(e_n+a_{n+1,n}e_{n+1})=xe_{n+1},
    \end{align*}
    so $x=0$. Therefore, $\vf(e_1)=e_1+a_{n+1,1}e_{n+1}$.  Finally, using $\vf(e_i)=e_i+a_{n+1,i}e_{n+1}$ and $e_i^2=e_{i+1}$ for all $1\le i\le n-1$ we get
    \begin{align*}
        e_{i+1}=(e_i+a_{n+1,i}e_{n+1})^2=\vf(e_i)^2=\vf(e_i^2)=\vf(e_{i+1})=e_{i+1}+a_{n+1,i+1}e_{n+1}.
    \end{align*}
    Therefore, $a_{n+1,i}=0$ for all $2\le i\le n$. Thus, the matrix of $\vf$ relative to the basis $(e_i)_{i=1}^{n+1}$ has the form $I+a_{n+1,1}E_{n+1,1}$.
     
    Conversely, the linear map $\vf$ defined by $\vf(e_1)=e_1+xe_{n+1}$, $\vf(e_i)=e_i$, $2\le i\le n+1$, is an automorphism of $\cA$, by Lemma~\ref{I+xE_1_n-auto}.
\end{proof}

\begin{lemma}\label{constructing-R_n+1}
Let $n\ge 3$ and $\cA\in\nil^{\gamma}_n$ such that  $\ann\cA=\la e_{n}\ra$ and $e_i^2=e_{i+1}$, for all $1\le i\le n-1$. Suppose that $\aut(\cA)=\{I+xE_{n1}\mid x\in\Co\}$ in the basis $(e_i)_{i=1}^n$. Then, there is a parametric family of pairwise non-isomorphic $1$-dimensional central extensions $\cB$ of $\cA$ with basis $(e_i)_{i=1}^{n+1}$, extending the basis of $\cA$, such that $\ann\cB=\la e_{n+1}\ra$, $e_1e_n=0$, $e_i^2=e_{i+1}$, for all $1\le i\le n$, and the structure constants $c_{ij}^{n+1}$ of $\cB$ in this basis are arbitrary independent complex parameters for all $1\le i\ne j\le n$, $(i,j)\ne (1,n)$. 
\end{lemma}
\begin{proof}
    We have
    \begin{align*}
    {\rm B}^2(\cA)=\la\Dt 11\ra\oplus\bigoplus_{m=2}^{n-1}\La\Dt mm+\sum_{1\le i\ne j\le m}c_{ij}^{m+1}\Dt ij\Ra,
    \end{align*}
    so ${\rm H}^2(\cA)=\la\Dl ij\mid 1\le i\ne j\le n\ra\oplus\la\Dl nn\ra$.
    
    Let $\phi=I+xE_{n1}\in\aut(\cA)$ and $\0=\sum_{1\le i\ne j\le n}\af_{ij}\Dl ij+\af_{nn}\Dl nn\in {\rm B}^2(\cA)$. Consider the corresponding matrix $A=\sum_{1\le i\ne j\le n}\af_{ij}E_{ij}+\af_{nn}E_{nn}$. Then
    \begin{align}
        \phi^TA\phi&=(I+xE_{1n})\left(\sum_{1\le i\ne j\le n}\af_{ij}E_{ij}+\af_{nn}E_{nn}\right)(I+xE_{n1})\notag\\
        &=A+x\sum_{i=1}^n\af_{ni}E_{1i}+x\sum_{i=1}^n\af_{in}E_{i1}+x^2\af_{nn}E_{11}\notag\\
        &=x(\af_{n1}+\af_{1n}+x\af_{nn})E_{11}+\sum_{i=2}^n(\af_{1i}+x\af_{ni})E_{1i}+\sum_{i=2}^n(\af_{i1}+x\af_{in})E_{i1}\notag\\
        &\quad+\sum_{2\le i\ne j\le n}\af_{ij}E_{ij}+\af_{nn}E_{nn}.\label{phi^T.A.phi}
    \end{align}
    So, $\phi\cdot\0=\sum_{1\le i\ne j\le n}\af^*_{ij}\Dl ij+\af^*_{nn}\Dl nn$, where $\af^*_{1i}=\af_{1i}+x\af_{ni}$, $\af^*_{i1}=\af_{i1}+x\af_{in}$, for all $2\le i\le n$, $\af^*_{ij}=\af_{ij}$, for all $2\le i\ne j\le n$, and $\af^*_{nn}=\af_{nn}$.  
    
    If $\af_{nn}\ne 0$, then choosing $x=-\af_{1n}\af_{nn}^{-1}$, we obtain the family of representatives of distinct orbits
    $$
    \La\sum_{1\le i\ne j\le n,(i,j)\ne(1,n)}c_{ij}^{n+1}\Dl ij+\Dl nn\Ra_{c_{ij}^{n+1}\in\Co}.
    $$ 
    It determines the desired family of algebras $\cB$. 
\end{proof}

\begin{definition}\label{R_n-defn}
Let $n\ge 3$. Denote by $\R_n$ the family of nilpotent algebras with basis $(e_i)_{i=1}^{n}$ satisfying~\eqref{g_ij=0-for-k<=max(ij)}, such that 
$e_i^2=e_{i+1}$, for all $1\le i\le n-1$, $c_{21}^3=1$, $c_{1i}^{i+1}=0$, for all $2\le i\le n-1$, and with the remaining structure constants $c_{ij}^{k}$ being arbitrary independent complex parameters, for all $1\le i\ne j\le n$, $k>\max\{i,j\}$.
\end{definition}

Notice that, given $\mathcal{A}\in\R_{n+1}$ with $n\ge 3$, since by~\eqref{g_ij=0-for-k<=max(ij)} we have $e_{n+1}\in\ann{\mathcal{A}}$, it follows that $\la e_{n+1}\ra$ is an ideal of $\mathcal{A}$ and $\mathcal{A}/\la e_{n+1}\ra$ can be seen naturally as an element of $\R_{n}$, relative to the ordered basis $(\overline{e_i})_{i=1}^{n}$, where $\overline{e_i}=e_i+\la e_{n+1}\ra$. This property will be important in arguments by induction, as the one which follows.
   
\begin{lemma}\label{L:ann-aut}
Let $\mathcal{A}\in\R_{n}$ with $n\geq 3$. Then the following hold:
\begin{enumerate}[(a)]
\item $\ann{\mathcal{A}}=\la e_{n}\ra$.
\item $\aut{\mathcal{A}}=\{I+xE_{n1}\mid x\in\Co\}$, relative to the basis $(e_i)_{i=1}^n$.
\end{enumerate}
\end{lemma}
\begin{proof}
We prove both statements simultaneously by induction on $n\geq 3$. Indeed, $\R_3$ consists of a single point, which as an algebra is defined by the following multiplication (as usual, only nonzero products of the basis elements are shown):
\begin{equation*}
\mathcal{A}:\ e_1 e_1 = e_2, \quad e_2 e_1= e_3, \quad e_2 e_2=e_3.
\end{equation*}
It is easy to see by direct inspection that $\ann{\mathcal{A}}=\la e_{3}\ra$ and $\aut{\mathcal{A}}=\{I+xE_{31}\mid x\in\Co\}$ in $(e_i)_{i=1}^{3}$. 

Now let $\mathcal{A}\in\R_{n+1}$. Then, viewing $\mathcal{A}/\la e_{n+1}\ra$ as an algebra from $\R_n$, as explained above, the induction hypothesis implies that $\ann{\cA/\la e_{n+1}\ra}=\la \overline{e_{n}}\ra$. Let $v=\sum_{i=1}^{n+1}\lambda_i e_i\in\ann \cA$. Since $e_{n+1}\in\ann \cA$, we can assume that $\lambda_{n+1}=0$ and our goal is to show that $\lambda_i=0$ for all $i$. Since $\overline v=\sum_{i=1}^{n}\lambda_i \overline{e_i}\in\ann{\cA/\la e_{n+1}\ra}$, we deduce that $\lambda_i=0$ for all $i\leq n-1$.
Thus, $\lambda_n e_n\in\ann \cA$, so $0=\lambda_n e_n e_n=\lambda_n e_{n+1}$ and $\lambda_n=0$, proving our claim that $\ann \cA=\la e_{n+1}\ra$.

The induction hypothesis also gives that $\aut {(\cA/\la e_{n+1}\ra)}=\{I+xE_{n1}\mid x\in\Co\}$ in $(\overline{e_i})_{i=1}^n$. Then, by Lemma~\ref{Aut(A)-exactly-I+xE_n1}$, \aut \cA=\{I+xE_{n+1,1}\mid x\in\Co\}$ in the basis $(e_i)_{i=1}^{n+1}$.
\end{proof}

\begin{proposition}\label{P:dim:nilp:Rn}
We have $\dim\left(\overline{\bigcup_{\cA\in\R_n}O(\cA)}\right)=\frac{n(n-1)(n+1)}3$. 
\end{proposition}
\begin{proof}
For any $\cA\in\R_n$, we know that $\dim\aut \cA=1$, by Lemma~\ref{L:ann-aut}, and thus $\dim O(\cA)=\dim\operatorname{GL}_n(\Co)-1=n^2-1$. Moreover, the algebras in $\R_n$ are pairwise non-isomorphic, by Lemma~\ref{constructing-R_n+1}, so the corresponding orbits are disjoint. 

Let $p_n=\dim \R_n$, the number of independent parameters of the family $\R_n$. We calculate $p_n$ by induction on $n$. We have $p_3=0$ and $p_{n+1}=p_n+n(n-1)-1$, for all $n\ge 3$. Therefore, $p_n=\frac{(n^2-1)(n-3)}3$. Thus, $\dim\left(\overline{\bigcup_{\cA\in\R_n}O(\cA)}\right)=n^2-1+p_n=\frac{n(n-1)(n+1)}{3}$.
\end{proof}

Before our main result of this section, we need a technical observation on the inverse of a lower triangular matrix. 

\begin{lemma}\label{a'_ii-in-terms-of-a_ij}
	Let $A=(a_{ij})_{i,j=1}^n$ be an invertible lower triangular matrix of size $n$ and $A^{-1}=(a'_{ij})_{i,j=1}^n$. Then for all $i> j$ we have $a'_{ij}=-a_{ii}^{-1}a_{jj}^{-1}a_{ij}-a_{jj}^{-1}\sum_{k=j+1}^{i-1}a'_{ik}a_{kj}$. In particular, for all $i\geq j$, $a'_{ij}$ is uniquely determined by $a_{pq}$ with $i\ge p\ge q\ge j$. 
\end{lemma}
\begin{proof}
	If $i=j$, then $a'_{ij}=a_{ij}^{-1}$. Otherwise, $\sum_{k=j}^i a'_{ik}a_{kj}=0$, whence $a'_{ij}=-a_{jj}^{-1}\sum_{k=j+1}^ia'_{ik}a_{kj}=-a_{ii}^{-1}a_{jj}^{-1}a_{ij}-a_{jj}^{-1}\sum_{k=j+1}^{i-1}a'_{ik}a_{kj}$. The second statement follows by backward induction on $j$ with $i$ fixed.
\end{proof}

Now we can state and prove our main result about the variety $\nil_n$ of nilpotent algebras. Notice that the proof specifically gives an algorithmic construction of a degeneration $\R_n\to N$ from the family $\R_n$ to any given $n$-dimensional nilpotent algebra $N$.

\begin{theorem}\label{T:Rn:generic}
	For any $n\ge 3$, the family $\R_n$ is generic in $\nil_n$. In particular, $\dim(\nil_n)=\frac{n(n-1)(n+1)}{3}$.
\end{theorem}
\begin{proof}
Given $N\in\nil_n$, we will prove that $\R_n\to N$. Recall that $(e_i)_{i=1}^n$ is our fixed basis of the underlying vector space. Without loss of generality, we may assume that $N\in\nil_n^\gamma$. Indeed, by Lemma~\ref{Nil^gm}, $N$ is isomorphic to an algebra $M$ whose structure constants satisfy~\eqref{g_ij=0-for-k<=max(ij)} in some basis $(f_i)_{i=1}^n$. Take $g\in\operatorname{GL}_n(\Co)$ such that $g(f_i)=e_i$, for all $1\le i\le n$. Then the structure constants of $gM$ in $(e_i)_{i=1}^n$ are those of $M$ in $(f_i)_{i=1}^n$, and thus satisfy~\eqref{g_ij=0-for-k<=max(ij)}. So, we may replace $N$ by $gM$, if necessary.
	 
	 We will thus assume that $N\in\nil_n^\gamma$ and prove by induction on $n$ that there is a parametric basis $E_i(t)=\sum_{j=i}^n a_{ji}(t)e_j$, with $a_{ji}(t)\in\Co(t)$, $1\le i\leq j\le n$, and a choice of structure constants $c_{ij}^k(t)\in\Co(t)$, satisfying the conditions of Definition~\ref{R_n-defn}, with
	\begin{align}\label{range-of-(i_j_k)}
		1\le i\ne j\le n,\ k>\max\{i,j\},\ (i,j,k)\ne(2,1,3),\ (i,j,k)\ne(1,k-1,k),
	\end{align}  
	giving a degeneration of $N$ from $\R_n$. 
	
	The case $n=3$ is proved in \cite{fkkv}. Let $N\in\nil_{n+1}^\gamma$. It follows that $e_{n+1}\in\ann N$, so $\la e_{n+1}\ra$ is an ideal of $N$ and $N/\la e_{n+1}\ra$ is seen as an element of $\nil_n^\gm$ via the identification of $e_i+\la e_{n+1}\ra$ with $e_i$, $1\le i\le n$. By the induction hypothesis, there is a parametric basis $E_i(t)=\sum_{j=i}^n a_{ji}(t)e_j$, $1\le i\le n$, and a choice of parameters $c_{ij}^k(t)$ satisfying the conditions of Definition~\ref{R_n-defn} and determining a degeneration of $N/\la e_{n+1}\ra$ from $\R_n$. Observe that the degeneration does not depend on $a_{n1}(t)$ because $e_n\in\ann R$ for all $R\in\R_n$. More generally, since any such $R$ satisfies~\eqref{g_ij=0-for-k<=max(ij)}, we have
	\begin{align}
		E_i(t)E_j(t)&=\sum_{p=i,q=j}^n a_{pi}(t)a_{qj}(t)e_pe_q
		=\sum_{p=i,q=j}^n a_{pi}(t)a_{qj}(t)\sum_{r>\max\{p,q\}}c_{pq}^r(t)e_r\notag\\
		&=\sum_{r=2}^n\left(\sum_{p=i,q=j}^{r-1}c_{pq}^r(t)a_{pi}(t)a_{qj}(t)\right)e_r.\label{E_iE_j=sum-e_m}
	\end{align}
	We see that $a_{n1}(t)$ cannot appear among the $a_{pi}(t)$ or $a_{qj}(t)$ above. Moreover, each $e_r$ from the sum~\eqref{E_iE_j=sum-e_m} belongs to $\la e_2,\dots,e_n\ra=\la E_2(t),\dots,E_n(t)\ra$, so its coordinates in the basis $(E_i(t))_{i=1}^n$ do not depend on $a_{i1}(t)$, $1\le i\le n$.
	 
	We are going to redefine $a_{n1}(t)$ and choose $a_{n+1,i}(t)$, $1\le i\le n+1$, with $a_{n+1,n+1}(t)\neq 0$, and $c_{ij}^{n+1}(t)$, $1\le i\ne j\le n$, $(i,j)\ne(1,n)$, such that $\tilde E_i(t):=E_i(t)+a_{n+1,i}(t)e_{n+1}$, $1\le i\le n$, $\tilde E_{n+1}(t):=a_{n+1,n+1}(t)e_{n+1}$ is a parametric basis giving a degeneration of $N$ from $\R_{n+1}$. Denote by $A(t)$ the lower triangular matrix $(a_{ij}(t))_{i,j=1}^{n+1}$ whose $(n+1)$-st row consists of unknown parameters which will be defined below and let $A^{-1}(t)=(a'_{ij}(t))_{i,j=1}^{n+1}$ be its formal inverse. Observe that the upper left $(n\times n)$-block of $A^{-1}(t)$ is the inverse of the upper left $(n\times n)$-block of $A(t)$ and thus does not depend on the choice of $a_{n+1,i}(t)$, $1\le i\le n+1$. Since the coordinates of $e_i$ in the basis $(\tilde E_j(t))_{j=1}^{n+1}$ are given by the $i$-th column of $A^{-1}(t)$, for all $1\le i\le n+1$, we can further develop~\eqref{E_iE_j=sum-e_m} to get
\begin{align}
		\tilde E_i(t)\tilde E_j(t)&=\sum_{k=2}^{n+1}\left(\sum_{r=2}^k a'_{kr}(t)\sum_{p=i,q=j}^{r-1}c_{pq}^r(t)a_{pi}(t)a_{qj}(t)\right)\tilde E_k(t),\label{new-str-constants}
\end{align}
for all $1\le i,j\le n+1$. Notice that these new structure constants satisfy~\eqref{g_ij=0-for-k<=max(ij)}.
	
Let $\gm_{ij}^k$ be the structure constants of $N$ in $(e_i)_{i=1}^{n+1}$. Thence, to construct the degeneration $\R_{n+1}\to N$, we need that
	\begin{align}\label{system-deg}
		\lim_{t\to 0}\left(\sum_{r=2}^{k}a'_{kr}(t)\sum_{p=i,q=j}^{r-1}c_{pq}^r(t)a_{pi}(t)a_{qj}(t)\right)=\gm_{ij}^k,\quad 1\le i,j<k\le n+1,
	\end{align}
be satisfied. Observe that~\eqref{system-deg} holds for all $1\le i,j<k\le n$ by the choice of $(E_i(t))_{i=1}^n$ and $(c_{ij}^k(t))$ with~\eqref{range-of-(i_j_k)}, because $\gm_{ij}^k$ is the corresponding structure constant of $N/\la e_{n+1}\ra$ for such $(i,j,k)$. Thus, it remains to consider $k=n+1$, which we do below by appropriately defining $a_{n1}(t)$, $a_{n+1,i}(t)$, $1\le i\le n+1$, and $c_{ij}^{n+1}(t)$, $1\le i\ne j\le n$, $(i,j)\ne(1,n)$ (so that the conditions of Definition~\ref{R_n-defn} hold).

We will proceed in $n$ steps, from $k=0$ to $k=n-1$. At the end of {Step $\mathbf{k}$} we will have defined $c_{p,q}^{n+1}(t)$ for all $p, q\geq n-k$ and $a_{n+1,r}(t)$ for all $r\geq n+1-k$. We will also have established the convergence 
\begin{align}\label{system-deg-k-is-n+1}
		\lim_{t\to 0}\left(\sum_{r=2}^{n+1}a'_{n+1,r}(t)\sum_{p=i,q=j}^{r-1}c_{pq}^r(t)a_{pi}(t)a_{qj}(t)\right)=\gm_{ij}^{n+1}
	\end{align}
for all $n\geq i, j\geq n-k$.
	
	\medskip
	
	\textbf{Step $\mathbf 0$.} Since we must have $c_{nn}^{n+1}(t)=1$, it remains to define $a_{n+1,n+1}(t)$. The left-hand side of~\eqref{system-deg-k-is-n+1} for $i=j=n$ becomes $\lim_{t\to 0}\left(a_{nn}(t)^2a_{n+1,n+1}(t)^{-1}\right)$, so we set
	\begin{align*}
		a_{n+1,n+1}(t):=
		\begin{cases}
			(\gm_{nn}^{n+1})^{-1}a_{nn}(t)^2, & \text{if $\gm_{nn}^{n+1}\ne 0$,}\\
			t^{-1}a_{nn}(t)^2, & \text{if $\gm_{nn}^{n+1}=0$.}
		\end{cases}
	\end{align*}
By definition, \eqref{system-deg-k-is-n+1} holds for $i=j=n$. Notice also that $a_{n1}(t)$ does not occur in the formula above.

\medskip

\textbf{Step $\mathbf{k}$.} Let $1\leq k< n-1$ and assume that {Step $\mathbf{k-1}$} has been successfully completed and that $a_{n1}(t)$ has not been used to define any new coefficients.

Suppose first that $n\geq i>j=n-k$. We will define $c_{ij}^{n+1}(t)$ and establish~\eqref{system-deg-k-is-n+1} in this case. The coefficient of $c_{ij}^{n+1}(t)$ on the left-hand side of~\eqref{system-deg-k-is-n+1} equals $a_{n+1,n+1}(t)^{-1}a_{ii}(t)a_{jj}(t)$ which has already been defined and is non-zero. We thus put
	\begin{align}\label{c_ij^(n+1)=}
		c_{ij}^{n+1}(t):=\frac{a_{n+1,n+1}(t)}{a_{ii}(t)a_{jj}(t)}\left(\gm_{ij}^{n+1}-\sum_{r=2}^{n+1}a'_{n+1,r}(t)\sideset{}{'}\sum_{p=i,q=j}^{r-1}c_{pq}^r(t)a_{pi}(t)a_{qj}(t)\right),
	\end{align}
where the primed sum is over all $(p,q)$ such that $(p,q,r)\ne(i,j,n+1)$. Note that on the right-hand side of~\eqref{c_ij^(n+1)=} we must have $n-k<i\leq r-1$, so $r\geq n-k+2$. Thence, by {Step $\mathbf{k-1}$} and Lemma~\ref{a'_ii-in-terms-of-a_ij}, all the terms of the form $a'_{n+1,r}(t)$ on the right-hand side of~\eqref{c_ij^(n+1)=} have already been defined. The same holds for all remaining terms except those of the form $c_{pj}^{n+1}(t)$ with $p> i$. Thus, \eqref{c_ij^(n+1)=} is a recurrence formula which defines $c_{ij}^{n+1}(t)$ in terms of $c_{pj}^{n+1}(t)$ with $p>i$. So, starting recursively with $c_{nj}^{n+1}(t)$, we can define all of the terms $c_{p,n-k}^{n+1}(t)$, with $p>n-k$ and by doing so we force the convergence~\eqref{system-deg-k-is-n+1} for all $i>n-k$ and $j=n-k$. Similarly, we can define all the terms $c_{n-k,q}^{n+1}(t)$, with $q>n-k$ recursively, making sure that~\eqref{system-deg-k-is-n+1} holds for $i=n-k$ and all $j>n-k$. This will work as before because we are assuming that $k<n-1$ so $(n-k,q)\neq (1,n)$. Moreover, also by that assumption on $k$, the coefficient $a_{n1}(t)$ has not been used in~\eqref{c_ij^(n+1)=} to define $c_{ij}^{n+1}(t)$, as $i, j\geq n-k\geq 2$. Hence, given that $c_{n-k,n-k}^{n+1}(t)=0$, all $c_{p,q}^{n+1}(t)$ with $p, q\geq n-k$ are defined and~\eqref{system-deg-k-is-n+1} holds for all $i, j\geq n-k$, except in the case $i=j=n-k$, which will be analyzed next.

Now we will define $a_{n+1,n+1-k}(t)$ so that~\eqref{system-deg-k-is-n+1} holds for $i=j=n-k$. Assume thus that $i=j=n-k$. Using Lemma~\ref{a'_ii-in-terms-of-a_ij} we have
\begin{align*}
	\sum_{r=2}^{n+1}\sum_{p,q=i}^{r-1}a'_{n+1,r}(t)c_{pq}^r(t)a_{pi}(t)a_{qi}(t)&=a'_{n+1,i+1}(t)a_{ii}(t)^2+\sum_{r=i+2}^{n+1}a'_{n+1,r}(t)\sum_{p,q=i}^{r-1}c_{pq}^r(t)a_{pi}(t)a_{qi}(t)\\
	&=-a_{n+1,n+1}(t)^{-1}a_{i+1,i+1}(t)^{-1}a_{ii}(t)^2a_{n+1,i+1}(t)\\
	&\quad-a_{i+1,i+1}(t)^{-1}a_{ii}(t)^2\sum_{s=i+2}^n a'_{n+1,s}(t)a_{s,i+1}(t)\\
	&\quad+\sum_{r=i+2}^{n+1}a'_{n+1,r}(t)\sum_{p,q=i}^{r-1}c_{pq}^r(t)a_{pi}(t)a_{qi}(t).
	\end{align*}
	Hence, we put
	\begin{align}
	a_{n+1,i+1}(t):=&-\frac{a_{n+1,n+1}(t)a_{i+1,i+1}(t)}{a_{ii}(t)^2}\left(\gm_{ii}^{n+1}-\sum_{r=i+2}^{n+1}a'_{n+1,r}(t)\sum_{p,q=i}^{r-1}c_{pq}^r(t)a_{pi}(t)a_{qi}(t)\right)\notag\\
	&\quad-a_{n+1,n+1}(t)\sum_{s=i+2}^{n}a'_{n+1,s}(t)a_{s,i+1}(t),\label{a_(n+1_n+1)(t)=}
	\end{align}
where the right-hand side defines $a_{n+1,n-k+1}(t)$ in terms of $a_{n+1,r}(t)$ with $r\geq n-k+2$ (already defined in the previous steps) and $c_{p,q}^{n+1}(t)$ with $p, q\geq n-k$ (defined above). Also, \eqref{system-deg-k-is-n+1} holds for $i=j=n-k$ and $a_{n1}(t)$ does not occur in the definition~\eqref{a_(n+1_n+1)(t)=} above. This step is thus finished.

\medskip

\textbf{Step $\mathbf{n-1}$.} When we reach this final step, all $c_{p,q}^{n+1}(t)$ with $p, q\geq 2$ and all $a_{n+1,r}(t)$ with $r\geq 3$ have been defined without using the coefficient $a_{n1}(t)$ and~\eqref{system-deg-k-is-n+1} has been shown to hold for all $i, j\geq 2$. Hence, as $a_{n1}(t)$ also has no role in \eqref{E_iE_j=sum-e_m} nor on the invertibility of $A(t)$, we can redefine it at this point. We will do it so as to guarantee that~\eqref{system-deg-k-is-n+1} holds for $(i,j)=(1,n)$. This is necessary because we are bound to having $c^{n+1}_{1n}(t)=0$, so we cannot force~\eqref{system-deg-k-is-n+1} in case $(i,j)=(1,n)$ by choosing $c^{n+1}_{1n}(t)$ as we please. 

Suppose thus that $(i,j)=(1,n)$. We have 
	\begin{align*}
		\sum_{r=2}^{n+1}\sum_{p=i,q=j}^{r-1}a'_{n+1,r}(t)c_{pq}^r(t)a_{pi}(t)a_{qj}(t)=a_{n+1,n+1}(t)^{-1}a_{nn}(t)\sum_{p=2}^{n}c_{pn}^{n+1}(t)a_{p1}(t),
	\end{align*}
	in which the coefficient of $a_{n1}(t)$ is $a_{n+1,n+1}(t)^{-1}a_{nn}(t)\ne 0$. Hence we set
	\begin{align*}
		a_{n1}(t):=\frac{a_{n+1,n+1}(t)\gm_{1n}^{n+1}}{a_{nn}(t)}-\sum_{p=2}^{n-1}c_{pn}^{n+1}(t)a_{p1}(t),
	\end{align*}
the right-hand side of which has already been defined and does not involve $a_{n1}(t)$. 

Now we can proceed as in the previous (generic) step with $k=n-1$, defining $c_{p1}^{n+1}(t)$ for $p\geq 2$ and then $c_{1q}^{n+1}(t)$ for $n-1\geq q\geq 2$ and finally $a_{n+1,2}(t)$, ensuring that~\eqref{system-deg-k-is-n+1} holds in the remaining cases. The coefficient $a_{n+1,1}(t)$ is unrestrained and can be chosen arbitrarily (which agrees with our previous observations).

This finishes the construction and the proof.
\end{proof}

\section{The geometric classification of commutative nilpotent algebras}\label{SS:commutative nilpotent algebras}

In this section, we consider the variety of commutative nilpotent $n$-dimensional algebras. Our methods will be analogous to those of Section~\ref{SS:nilpotent algebras}.

\begin{definition}\label{S_n-defn}
Let $n\ge 4$. Denote by $\cS_n$ the family of commutative algebras in $\nil^{\gamma}_n$ such that 
$e_i^2=e_{i+1}$ for all $1\le i\le n-1$, $c_{23}^4=1$, $c_{12}^4\neq 0$ and $c_{1i}^{i+1}=0$ for all $2\le i\le n-1$. The remaining structure constants $c_{ij}^{k}$ are arbitrary, subject only to~\eqref{g_ij=0-for-k<=max(ij)} and the commutativity constraint.
\end{definition}

As with the algebras $\R_n$ from Definition~\ref{R_n-defn}, if $\cA\in\cS_{n+1}$, for some $n\ge 4$, then $\cA/\la e_{n+1}\ra$ is seen naturally as an element of $\cS_{n}$, relative to the ordered basis $(e_i+\la e_{n+1}\ra )_{i=1}^{n}$.

\begin{example}\label{Ex:C:S4}
Let $n=4$ and $\alpha\in\Co$. Define the commutative algebra $\cA_\alpha$ by the multiplication
\begin{equation}\label{E:C:S4}
\cA_\alpha:\ e_1^2=e_2,\quad e_1 e_2=\alpha e_4,\quad e_1 e_3=0,\quad e_2^2=e_3,\quad e_2 e_3= e_4,\quad e_3^2=e_4,
\end{equation}
where $e_4\in \ann{\cA_\alpha}$. Then $\cS_4$ consists of the algebras $\cA_\alpha$ with $\alpha\in\Co^*$. Considering the new basis 
\begin{equation*}
f_1=e_1-\alpha e_3,\quad f_2=e_2+\alpha^2 e_4,\quad f_3=e_3\quad\text{and}\quad f_4=e_4,
\end{equation*}
we see that $\cA_\alpha$ is isomorphic to the algebra $\mathcal{C}_{19}(-\alpha)$ defined in \cite{fkkv}. 

Recall that in \cite[Thm.\ 5]{fkkv} it was shown that the family $\mathcal{C}_{19}(\alpha)$, with $\alpha\in\Co$, is generic in the variety of $4$-dimensional nilpotent commutative algebras. Since $\cA_0\in\overline{\bigcup_{\alpha\in\Co^*} O(\cA_\alpha)}$, it follows that the family $\cS_4$ is also generic in the variety of $4$-dimensional commutative nilpotent algebras. 
\end{example}

As we will see next, our restriction in Definition~\ref{S_n-defn} that $c_{12}^4\neq 0$ ensures that the algebras in $\cS_4$ have a sufficiently small automorphism group. 

\begin{lemma}\label{L:S:ann-aut}
Let $\cA\in\cS_{n}$. Then the following hold:
\begin{enumerate}[(a)]
\item $\ann{\cA}=\la e_{n}\ra$.
\item $\aut{\cA}=\{I+xE_{n1}\mid x\in\Co\}$, relative to the basis $(e_i)_{i=1}^n$.
\end{enumerate}
\end{lemma}
\begin{proof}
The proof is essentially the same as that of Lemma~\ref{L:ann-aut}, since Lemmas~\ref{I+xE_1_n-auto} and~\ref{Aut(A)-exactly-I+xE_n1} will still apply to the algebras in $\cS_n$. We just need to verify the base cases for (a) and (b). Assume thus that $n=4$. Then $\cS_4$ is described in Example~\ref{Ex:C:S4}; more specifically, it consists of the algebras $\cA_\alpha$ with $\alpha\neq 0$ and multiplication given by \eqref{E:C:S4}, commutativity and the fact that $e_4\in \ann{\cA_\alpha}$. 

Let $v=\sum_{i=1}^4 \lambda_i e_i\in\ann{\cA_\alpha}$. As $e_4\in \ann{\cA_\alpha}$, we can assume that $\lambda_4=0$. Then 
\begin{align*}
&0=ve_1=\lambda_1e_2+\lambda_2\alpha e_4, \quad\text{so $\lambda_1=0$;}\\
&0=ve_2=\lambda_2e_3+\lambda_3e_4, \quad\text{so $\lambda_2=\lambda_3=0$.}
\end{align*}
So indeed $\ann{\cA_\alpha}=\la e_{4}\ra$, for every $\alpha$.

Now, for (b), Lemma~\ref{I+xE_1_n-auto} guarantees that $\aut {\cA_\alpha}\supseteq\{I+xE_{n1}\mid x\in\Co\}$. Conversely, let $\varphi\in\aut {\cA_\alpha}$, with matrix $\phi=\left(a_{ij}\right)_{1\leq i,j\leq 4}$ relative to the ordered basis $(e_i)_{i=1}^4$. Since $\ann{\cA_\alpha}=\la e_{4}\ra$, it follows that $\varphi(e_4)=ze_4$, for some $z\in\Co^*$. So $\varphi$ induces an automorphism $\overline\varphi:\cA_\alpha/\la e_{4}\ra\longrightarrow \cA_\alpha/\la e_{4}\ra$ and, relative to the basis $\overline{e_i}=e_i+\la e_{4}\ra$, $i=1, 2, 3$, the nonzero products among basis vectors are just $\overline{e_i}^2=\overline{e_{i+1}}$, for $i=1, 2$. It is easy to see (cf.\ \cite[3.1.1]{fkkv}, where this $3$-dimensional algebra is denoted by $\mathcal{C}_{02}$) that the matrix of $\overline\varphi$ is of the form 
\begin{equation*}
\begin{pmatrix}
x&0&0\\ 0&x^2&0\\y&0&x^4 
\end{pmatrix},
\end{equation*}
for $x, y\in\Co$ with $x\neq 0$. Thus, we conclude that 
\begin{equation*}
\phi=
\begin{pmatrix}
x&0&0&0\\ 0&x^2&0&0\\y&0&x^4 &0\\
a_{41}&a_{42}&a_{43}&z
\end{pmatrix}.
\end{equation*}

Applying $\varphi$ to the relation $e_1e_3=0$ yields
\begin{equation*}
0=\varphi(e_1)\varphi(e_3)=(xe_1+ye_3+a_{41}e_4)(x^4e_3+a_{43}e_4)= yx^4e_4,
\end{equation*}
so $y=0$. Similarly, using $e_1^2=e_2$ we deduce that $a_{42}=0$; then $e_2^2=e_3$ gives $a_{43}=0$ and by $e_2e_3=e_4$ we get $z=x^6$. So it remains to show that $x=1$, which we do by using $e_3^2=e_4$ and $e_1e_2=\alpha e_4$, with $\alpha\neq0$. The former relation implies that $x^8=z=x^6$, so $x^2=1$, and then  
\begin{equation*}
\alpha e_4=\alpha x^6 e_4=\varphi(\alpha e_4)=(xe_1+a_{41}e_4)(x^2e_2)=x^3\alpha e_4=x\alpha e_4.
\end{equation*}
As $\alpha\neq 0$, we can deduce from the above that $x=1$.
\end{proof}

\begin{remark}
The proof of Lemma~\ref{L:S:ann-aut} shows also that, although $\ann{\cA_0}=\la e_{4}\ra$, the automorphism group of $\cA_0$ is slightly larger: $\aut{\cA_0}=\left\{ \left(\begin{smallmatrix}
\pm1&0&0&0\\ 0&1&0&0\\0&0&1 &0\\
x&0&0&1
\end{smallmatrix}\right)\mid x\in\Co\right\}$.
\end{remark}

Since we are now working in a variety of commutative algebras, we need to slightly adapt the method described in Subsection~\ref{SS:ceSS}. Specifically, for a commutative $n$-dimensional algebra ${\cA}$, let ${\rm Z}^2_{\mathcal C}\left( {\cA},\Co\right)$ 
be the subspace of ${\rm Z}^{2}\left( {\cA},\Co\right)$ consisting of the $2$-cocycles $\theta :{\cA}\times {\cA}\longrightarrow \Co$ such that $\theta(x,y)=\theta(y,x)$, for all $x, y\in\cA$. Then ${\rm B^2}(\cA,\Co)\subseteq {\rm Z}^2_{\mathcal C}\left( {\cA},\Co\right)$ and we set 
${\rm H}_\mathcal{C}^2(\cA,\Co) = {\rm Z}^2_{\mathcal C}\left( {\cA},\Co\right) \big/{\rm B^2}(\cA,\Co)$. This is a subspace of ${\rm H}^2(\cA,\Co)$ and $\cA_\theta$ is commutative if and only if 
$\theta\in{\rm Z}^2_{\mathcal C}\left( {\cA},\Co\right)$. We define $\Delta^c_{ij}=\Delta_{ij}+\Delta_{ji}$ and $\Delta^c_{ii}=\Delta_{ii}$, for $1\leq i\neq j\leq n$, so that $\left\{\Delta^c_{ij}\mid 1\leq i\leq j\leq n\right\}$ is a basis of ${\rm Z}_\mathcal{C}^2(\cA,\Co)$.

We have the following analogue of Lemma~\ref{constructing-R_n+1}.

\begin{lemma}\label{constructing-S_n+1}
Let $n\geq 4$ and $\cA\in\nil^{\gamma}_n$ be commutative such that $\ann\cA=\la e_{n}\ra$ and $e_i^2=e_{i+1}$ for all $1\le i\le n-1$. Suppose that $\aut(\cA)=\{I+xE_{n1}\mid x\in\Co\}$ in the basis $(e_i)_{i=1}^n$. Then there is a parametric family of pairwise non-isomorphic commutative $1$-dimensional central extensions $\cB$ of $\cA$ with basis $(e_i)_{i=1}^{n+1}$, extending the basis of $\cA$, such that $\ann\cB=\la e_{n+1}\ra$, $e_1e_n=0$, $e_i^2=e_{i+1}$ for all $1\leq i\leq n$, and the structure constants $c_{ij}^{n+1}$ of $\cB$ in this basis are arbitrary independent complex parameters for all $1\le i< j\le n$, $(i,j)\ne (1,n)$. 
\end{lemma}
\begin{proof}
The proof is identical to that of Lemma~\ref{constructing-R_n+1}, essentially replacing $\Delta_{ij}$ by $\Delta^c_{ij}$ and $i\neq j$ by $i<j$. For example, $H_{\mathcal C}^2(\cA)=\la\left[\Delta_{ij}^c\right]\mid 1\le i< j\le n\ra\oplus\la\left[\Delta_{nn}^c\right]\ra$.
\end{proof}

\begin{proposition}
Let $n\geq 4$. Then $\dim\left(\overline{\bigcup_{\cA\in\cS_n}O(\cA)}\right)=\frac{n(n-1)(n+4)}{6}$. 
\end{proposition}
\begin{proof}
This proof is just an adaptation of the proof of Proposition~\ref{P:dim:nilp:Rn}. Since $\dim\aut \cA=1$, for every $\cA\in\cS_n$, we have $\dim O(\cA)=n^2-1$. Moreover, we have observed in Example~\ref{Ex:C:S4} that $\cS_4=\left\{\cA_\alpha\mid \alpha\neq 0\right\}$ and, by \cite[3.1.3]{fkkv}, $\cA_\alpha\simeq \cA_{\alpha'}$ if and only if $\alpha'=\pm\alpha$. Thus, the isomorphism classes in $\cS_4$ form a $1$-parameter family and the isomorphism classes in $\cS_n$ are obtained by iterated $1$-dimensional central extensions of this family, as shown in Lemma~\ref{constructing-S_n+1}.

Let $q_n$ be the number of independent parameters of the family $\cS_n$. We have $q_4=1$ and $q_{n+1}=q_n+\frac{n(n-1)}{2}-1$, for all $n\ge 4$. Therefore, $q_n=\frac{n(n+1)(n-4)}{6}+1$. Thus, $\dim\left(\overline{\bigcup_{\cA\in\cS_n}O(\cA)}\right)=n^2-1+q_n=\frac{n(n-1)(n+4)}{6}$.
\end{proof}

\begin{theorem}\label{Sn:generic}
For any $n\ge 4$, the family $\cS_n$ is generic in the variety of all  $n$-dimensional commutative nilpotent algebras. In particular, that variety has dimension $\frac{n(n-1)(n+4)}{6}$.
\end{theorem}

\begin{proof}
The proof is identical to the proof of Theorem~\ref{T:Rn:generic}, the homologous result for the variety of all $n$-dimensional nilpotent algebras. Indeed, the base step for $n=4$ is given in Example~\ref{Ex:C:S4} and for the inductive step we just need to observe that if $c_{ij}^{r}=c_{ji}^{r}$ and $\gamma_{ij}^r=\gamma_{ji}^r$ for all $1\leq i,j, r\leq n+1$, then \eqref{system-deg-k-is-n+1} holds for the pair $(i,j)$ if and only if it holds for $(j,i)$.
\end{proof}

\section{The geometric classification of anticommutative nilpotent algebras}\label{SS:anticommutative nilpotent algebras}

In this section, we consider the variety of anticommutative nilpotent $n$-dimensional algebras. Our methods will be analogous to those of Sections~\ref{SS:nilpotent algebras} and~\ref{SS:commutative nilpotent algebras} but there will be some additional technical difficulties coming from larger automorphism groups in lower dimensions.

To shorten the coming statements, we make the following auxiliary definition.

\begin{definition}\label{Tprime_n-defn}
Let $n\ge 3$. Denote by $\cT'_n$ the family of anticommutative algebras in $\nil_n^\gamma$ such that $e_i e_{i+1}=e_{i+2}$ for all $1\le i\le n-2$. The remaining structure constants $c_{ij}^{k}$ are arbitrary, subject only to~\eqref{g_ij=0-for-k<=max(ij)} and the anticommutativity constraint.
\end{definition}

\begin{lemma}\label{L:T:ann-aut}
Let $\cA\in\cT'_n$, with $n\geq 3$. Then the following hold:
\begin{enumerate}[(a)]
\item $\ann \cA=\la e_{n}\ra$.
\item For all $\alpha, \beta\in\Co$, the linear map $\vf(e_1)=e_1+\alpha e_n$, $\vf(e_2)=e_2+\beta e_n$, $\vf(e_i)=e_i$, $3\le i\le n$, is an automorphism of $\cA$.
\end{enumerate}
\end{lemma}
\begin{proof}
The first statement follows easily by induction, as in the proof of Lemma~\ref{L:ann-aut}, and the second statement follows just as in the proof of Lemma~\ref{I+xE_1_n-auto}, using the anticommutativity of $\cA$.
\end{proof}


Our immediate goal is to prove the converse of the second part of Lemma~\ref{L:T:ann-aut}, for sufficiently large $n$ and given a few extra conditions on the structure constants $c_{ij}^k$. We will do this over a series of lemmas, providing just the key steps in the proofs.

\begin{lemma}\label{L:T:autT4}
Let $\cA\in\cT'_{4}$ with $c_{13}^{4}=0$. Then, relative to the basis $(e_i)_{i=1}^4$, we have
\begin{equation*}
\aut \cA=\left\{ \left(\begin{smallmatrix}x&a_{12}&0&0\\0&y&0&0\\ a_{31}&a_{32}&xy&0\\a_{41}&a_{42}&-a_{31}y&xy^2\end{smallmatrix} \right) \mid a_{ij}\in\Co, x, y\in\Co^* \right\} .
\end{equation*}
\end{lemma}
\begin{proof}
Let $\vf\in\aut \cA$ and assume that, relative to $(e_i)_{i=1}^4$, the matrix of $\vf$ is $\left(a_{ij}\right)_{1\leq i, j\leq 4}$. We know that $\ann \cA=\la e_{4}\ra$, so $\vf(e_4)=a_{44}e_4$, with $a_{44}\neq 0$. Moreover, $\cA/\la e_{4}\ra$ can be seen naturally as an element of $\cT'_{3}$ and $\vf$ induces an automorphism of $\cA/\la e_{4}\ra$ with matrix $\left(a_{ij}\right)_{1\leq i, j\leq 3}$ relative to the basis $(e_i+\la e_{4}\ra)_{i=1}^3$. The reasoning above then gives $a_{13}=a_{23}=0$ and $a_{33}\neq 0$.

Next we apply $\vf$ to the identity $e_1 e_3=0$, which follows from $c_{13}^4=0$ and~\eqref{g_ij=0-for-k<=max(ij)}, to get
\begin{align*}
0=(a_{11}e_1+a_{21}e_2+a_{31}e_3+a_{41}e_4)(a_{33}e_3+a_{43}e_4) =a_{21}a_{33}e_4.
\end{align*}
Thus, $a_{21}=0$. Below we list, for each identity in $\cA$, the corresponding relations we obtain when we apply $\vf$, as above.

\begin{center}
\begin{tabular}{l@{\qquad}l}
{\sc Identity}&{\sc relation}\\\hline
$e_1 e_3=0$ & $a_{21}=0$\\
$e_1 e_2=e_3$ & $a_{33}=a_{11}a_{22}$,\quad $a_{43}=-a_{31}a_{22}$\\
$e_2 e_3=e_4$ & $a_{44}=a_{33}a_{22}$
\end{tabular}
\end{center}
These show the direct inclusion in the statement. Since the listed relations comprise all relations in $\cA$, the reverse inclusion follows as well.
\end{proof}

Next, we look at the $n=5$ case.

\begin{lemma}\label{L:T:autT5}
Let $\cA\in\cT'_{5}$ such that $c_{13}^{4}=c_{14}^{5}=c_{24}^5=0$. Then, relative to the basis $(e_i)_{i=1}^5$, we have
\begin{equation*}
\aut \cA=\left\{ \left(\begin{smallmatrix}x&a_{12}&0&0&0\\0&y&0&0&0\\ 0&0&xy&0&0\\ a_{41}&a_{42}&0&xy^2&0\\
a_{51}&a_{52}&0&a_{54}&x^2 y^3
\end{smallmatrix} \right) \mid a_{ij}\in\Co, x, y\in\Co^*, a_{41}=c_{13}^5 x(1-y^2), a_{54}=xy(a_{12}c_{13}^5-a_{42}) \right\} .
\end{equation*}
\end{lemma}

\begin{proof}
Let $\vf\in\aut \cA$. Using, as before, the fact that $\ann \cA=\la e_{5}\ra$ and Lemma~\ref{L:T:autT4}, we conclude that the matrix of $\vf$ relative to the standard basis is of the form
\begin{equation*}
\left(\begin{smallmatrix}x&a_{12}&0&0&0\\0&y&0&0&0\\ a_{31}&a_{32}&xy&0&0\\ a_{41}&a_{42}&-a_{31}y&xy^2&0\\
a_{51}&a_{52}&a_{53}&a_{54}&a_{55}
\end{smallmatrix} \right),
\end{equation*}
with $x, y, a_{55}\neq 0$. The remaining relations follow, as in the proof of Lemma~\ref{L:T:autT4}, by applying $\vf$ to the identities in $\cA$. We summarize these below.

\begin{center}
\begin{tabular}{l@{\qquad}l}
{\sc Identity}&{\sc relation}\\\hline
$e_1 e_4=0$ & $a_{31}=0$\\
$e_2 e_4=0$ & $a_{32}=0$\\
$e_3 e_4=e_5$ & $a_{55}=x^2 y^3$\\
$e_1 e_2=e_3$ & $a_{53}=0$\\
$e_1 e_3=c_{13}^{5}e_5$ & $a_{41}=c_{13}^{5}x(1-y^2)$\\
$e_2 e_3=e_4$ & $a_{54}=xy(a_{12}c_{13}^5-a_{42})$
\end{tabular}
\end{center}
Thus, the direct inclusion in the statement follows and the reverse follows as well since we have used all the identities in $\cA$.
\end{proof}

\begin{lemma}\label{L:T:autT6}
Let $\cA\in\cT'_{6}$ such that $c_{13}^{4}=c_{14}^{5}=c_{15}^{6}=c_{24}^5=c_{25}^6=c_{14}^6=c_{24}^6=c_{13}^6=0$ and $c_{13}^5 c_{35}^6\neq 0$. Then, relative to the basis $(e_i)_{i=1}^6$, we have
\begin{equation*}
\aut \cA=\left\{ \left(\begin{smallmatrix}x&0&0&0&0&0\\0&1&0&0&0&0\\ 
0&0&x&0&0&0\\ 0&0&0&x&0&0\\
0&0&0&0&x^2 &0\\
a_{61}&a_{62}&0&0&0&x^3
\end{smallmatrix} \right) \mid a_{61}, a_{62}\in\Co, x\in\Co^* \right\}.
\end{equation*}
\end{lemma}
\begin{proof}
The proof follows the same pattern as before. So, if $\vf\in\aut \cA$, then the matrix of $\vf$ relative to the standard basis is of the form
\begin{equation*}
\left(\begin{smallmatrix}x&a_{12}&0&0&0&0\\0&y&0&0&0&0\\ 0&0&xy&0&0&0\\ a_{41}&a_{42}&0&xy^2&0&0\\
a_{51}&a_{52}&0&a_{54}&x^2 y^3&0\\
a_{61}&a_{62}&a_{63}&a_{64}&a_{65}&a_{66}
\end{smallmatrix} \right),
\end{equation*}
with $x, y\in\Co^*$, $a_{41}=c_{13}^5 x(1-y^2)$ and $a_{54}=xy(a_{12}c_{13}^5-a_{42})$.
Now we apply $\vf$ to the identities in $\cA$ and the result is summarized in what follows.

\begin{center}
\begin{tabular}{l@{\qquad}l}
{\sc Identity}&{\sc relation}\\\hline
$e_4 e_5=e_6$ & $a_{66}=x^3y^5$\\
$e_3 e_5=c_{35}^6 e_6$ & $c_{35}^6 (a_{66}-x^3y^4)=0$\\
$e_1 e_5=0$ & $a_{41}=0$\\
$e_2 e_5=0$ & $a_{42}=0$
\end{tabular}
\end{center}
Therefore, as $c_{35}^6\neq 0$ and $x, y\neq 0$, we get $y=1$ and $a_{66}=x^3$. Notice also that the above is consistent with our previously deduced relation $a_{41}=c_{13}^5 x(1-y^2)$, and we also get $a_{54}=xa_{12}c_{13}^5$. Proceeding as before, we obtain the following additional relations.

\begin{center}
\begin{tabular}{l@{\qquad}l}
{\sc Identity}&{\sc relation}\\\hline
$e_1 e_4=0$ & $a_{51}=0$\\
$e_2 e_4=0$ & $a_{52}=0$\\
$e_1 e_2=e_3$ & $a_{63}=0$\\
$e_2 e_3=e_4$ & $a_{64}=0$\\
$e_1 e_3=c_{13}^5 e_5$ & $a_{65}=0$\\
$e_3 e_4=e_5$ & $a_{65}=xa_{54}c_{35}^6$
\end{tabular}
\end{center}
Therefore, as $xc_{35}^6\neq 0$, we deduce that $a_{54}=0$. But we had $a_{54}=xa_{12}c_{13}^5$ and $c_{13}^5\neq 0$, so $a_{12}=0$. The proof is thus complete.
\end{proof}

\begin{lemma}\label{L:T:autT7}
Let $\cA\in\cT'_{7}$ such that $c_{13}^{4}=c_{14}^{5}=c_{15}^{6}=c_{24}^5=c_{25}^6=c_{14}^6=c_{24}^6=c_{13}^6=0$ and $c_{13}^5 c_{35}^6 c_{46}^7\neq 0$. Then, relative to the basis $(e_i)_{i=1}^7$, we have
\begin{equation*}
\aut \cA=\left\{ I+a_{71}E_{71}+a_{72}E_{72}  \mid a_{71}, a_{72}\in\Co \right\}.
\end{equation*}
\end{lemma}
\begin{proof}
The proof follows the ongoing pattern. So, if $\vf\in\aut \cA$, then the principal $6\times 6$ submatrix of the matrix of $\vf$ relative to the standard basis of $\cA$ is of the form given in the statement of Lemma~\ref{L:T:autT6} and $\vf(e_7)=a_{77}e_7$. We proceed as before listing the relations deduced from each of the identities in $\cA$.

\begin{center}
\begin{tabular}{l@{\qquad}l}
{\sc Identity}&{\sc relation}\\\hline
$e_5 e_6=e_7$ & $a_{77}=x^5$\\
$e_4 e_6=c_{46}^7 e_7$ & $(x^4-a_{77})c_{46}^7=0$
\end{tabular}
\end{center}
Since $xc_{46}^7\neq 0$, we deduce from the above that $x=1$.

\begin{center}
\begin{tabular}{l@{\qquad}l}
{\sc Identity}&{\sc relation}\\\hline
$e_4 e_5=e_6$ & $a_{76}=0$\\
$e_3 e_4=e_5$ & $a_{75}=0$\\
$e_2 e_5=c_{25}^7 e_7$ & $a_{62}=0$\\
$e_2 e_3=e_4$ & $a_{74}=0$\\
$e_1 e_5=c_{15}^7 e_7$ & $a_{61}=0$\\
$e_1 e_2=e_3$ & $a_{73}=0$
\end{tabular}
\end{center}
Therefore, $\vf$ is of the desired form, which proves the direct inclusion in the statement. For the reverse inclusion, thanks to Lemma~\ref{L:T:ann-aut}, we needn't check the remaining identities as any linear map with matrix of the form $I+a_{71}E_{71}+a_{72}E_{72}$, relative to the standard basis, is an automorphism of $\cA$.
\end{proof}

\begin{example}\label{ex:anticomm:tab}
Let $n=6$ and take $\cA\in\cT'_{6}$ such that $c_{13}^{4}=c_{14}^{5}=c_{15}^{6}=c_{24}^5=c_{25}^6=c_{14}^6=c_{24}^6=c_{13}^6=0$ and $c_{13}^5 c_{35}^6\neq 0$. Then $\cA$ depends just on the two nonzero parameters $\alpha=c_{13}^5$ and $\beta=c_{35}^6$ and we denote this algebra by $\cA(\alpha, \beta)$. Define a new basis for $\cA(\alpha, \beta)$ as follows:
\begin{align*}
E_1=\beta e_2,\quad E_2=-e_1-\alpha e_4,\quad E_3=\beta e_3,\quad E_4=\beta^2 e_4, \quad E_5=\beta^3 e_5,\quad E_6=\beta^5 e_6. 
\end{align*}
Then, we obtain
\begin{equation*}
E_1 E_2= (\beta e_2)(-e_1-\alpha e_4)=\beta e_1e_2=\beta e_3=E_3,\quad E_1 E_3=(\beta e_2)(\beta e_3)=\beta^2e_4=E_4,
\end{equation*}
and similar computations show that the multiplication in this basis is given by:
\begin{center}
\setlength{\tabcolsep}{12pt}
\begin{tabular}{lllll}
$E_1 E_2=E_3$, & $E_1 E_3=E_4$, & $E_1 E_4=0$, & $E_1 E_5=0$, & $E_1 E_6=0$, \\
$E_2 E_3=0$, & $E_2 E_4=0$, & $E_2 E_5=-\alpha/\beta^2 E_6$, & $E_2 E_6=0$, & $E_3 E_4=E_5$, \\
$E_3 E_5=E_6$, & $E_3 E_6=0$, & $E_4 E_5=E_6$, & $E_4 E_6=0$, & $E_5 E_6=0$. 
\end{tabular}
\end{center}
It follows that $\cA(\alpha, \beta)\simeq \mathbb{A}_{82}(-\alpha/\beta^2)$, where the family $\mathbb{A}_{82}(\gamma)$, for $\gamma\in\Co^*$, was defined in~\cite[Thm.\ 1]{kkl19} and shown to be generic in the variety of $6$-dimensional complex nilpotent anticommutative algebras in~\cite[Thm.\ 2]{kkl19}. Moreover, since the algebras $\left\{\mathbb{A}_{82}(\gamma)\right\}_{\gamma\in\Co^*}$ are pairwise non-isomorphic, it follows that $\cA(\alpha, \beta)\simeq \cA(\alpha', \beta')$ if and only if $\alpha{\beta'}^2=\alpha' \beta^2$. Thus, $\cA(\alpha, \beta)\simeq \cA(\alpha/\beta^2, 1)$ and we can assume without loss of generality that $\beta=c_{35}^6=1$. We conclude that the algebras $\left\{\cA(\alpha, 1)\right\}_{\alpha\in\Co^*}$ are pairwise non-isomorphic and generic in the variety of $6$-dimensional complex nilpotent anticommutative algebras.
\end{example}

The results above, along with Example~\ref{ex:anticomm:tab}, motivate the following auxiliary definition. 

\begin{definition}\label{hT_n-defn}
Let $n\ge 6$. Denote by $\hat\cT_n$ the family of those algebras in $\cT'_n$
%
such that $c_{13}^{4}=c_{14}^{5}=c_{15}^{6}=c_{24}^5=c_{25}^6=c_{14}^6=c_{24}^6=c_{13}^6=0$, $c_{35}^6=1$ and $c_{13}^5  c_{46}^7\neq 0$ (in case $n=6$, the latter condition should be replaced with $c_{13}^5\neq 0$).
\end{definition}

We can finally prove that the algebras in $\hat\cT_n$, with $n\geq 7$, have the smallest possible automorphism group among $n$-dimensional nilpotent anticommutative algebras.

\begin{proposition}\label{P:T:aut}
Let $n\geq 7$ and suppose that $\cA\in\hat\cT_n$. Then $\aut \cA=\{I+a_{n1}E_{n1}+a_{n2}E_{n2}  \mid a_{n1}, a_{n2}\in\Co\}$, relative to the basis $(e_i)_{i=1}^n$.
\end{proposition}
\begin{proof}
The proof is by induction on $n\geq 7$ and the base step has been settled in Lemma~\ref{L:T:autT7}. So suppose that the result holds for all algebras in $\hat\cT_n$ and let us take $\cA\in\hat\cT_{n+1}$, with $n\geq 7$, and $\vf\in\aut \cA$. Then, since $\ann \cA=\la e_{n+1}\ra$ we conclude that $\vf(e_{n+1})=a_{n+1, n+1}e_{n+1}$ with $a_{n+1, n+1}\in\Co^*$. In particular, $\vf$ induces an automorphism of $\cA/\la e_{n+1}\ra$. We can see $\cA/\la e_{n+1}\ra$ as an element of $\hat\cT_n$ via the basis $(e_i+\la e_{n+1}\ra)_{i=1}^n$ and the induction hypothesis implies that 
\begin{align*}
\vf(e_1)&=e_1+a_{n1}e_n+a_{n+1, 1}e_{n+1}, \\ 
\vf(e_2)&=e_2+a_{n2}e_n+a_{n+1, 2}e_{n+1} \quad\mbox{and}\\ 
\vf(e_i)&=e_i+a_{n+1, i}e_{n+1}, \quad \mbox{for all $3\leq i\leq n$.}
\end{align*}
So it remains to show that $a_{n+1, n+1}=1$ and $a_{n1}=a_{n2}=a_{n+1, i}=0$, for all $3\leq i\leq n$. As before, we have the following table.

\begin{center}
\begin{tabular}{l@{\qquad}l}
{\sc Identity}&{\sc relation}\\\hline
$e_{n-1} e_n=e_{n+1}$ & $a_{n+1, n+1}=1$\\
$e_{i} e_{i+1}=e_{i+2}$, for $3\leq i\leq n-2$& $a_{n+1, i+2}=0$, for $3\leq i\leq n-2$
\end{tabular}
\end{center}
The last relation shows that $a_{n+1, i}=0$ for all $5\leq i\leq n$. Using these we get the following additional relations, which conclude the proof.

\begin{center}
\begin{tabular}{l@{\qquad}l}
{\sc Identity}&{\sc relation}\\\hline
$e_1 e_{n-1}=c_{1,n-1}^{n}e_n+c_{1,n-1}^{n+1}e_{n+1}$ & $a_{n1}=0$\\
$e_2 e_{n-1}=c_{2,n-1}^{n}e_n+c_{2,n-1}^{n+1}e_{n+1}$ & $a_{n2}=0$\\
$e_{1} e_2=e_{3}$ & $a_{n+1,3}=0$\\
$e_{2} e_3=e_{4}$ & $a_{n+1,4}=0$
\end{tabular}
\end{center}
%
\end{proof}

In our next step we consider central extensions, as in Lemmas~\ref{constructing-R_n+1} and~\ref{constructing-S_n+1}. Since we are now working in the variety of anticommutative algebras, as in the previous section, we accordingly adapt the method described in Subsection~\ref{SS:ceSS}. So, for an anticommutative $n$-dimensional algebra ${\cA}$, ${\rm Z}^2_{\mathcal{AC}}\left( {\cA},\Co\right)$ is the subspace of ${\rm Z}^{2}\left( {\cA},\Co\right)$ consisting of the $2$-cocycles $\theta :{\cA}\times {\cA}\longrightarrow \Co$ such that $\theta(x,y)=-\theta(y,x)$, for all $x, y\in{\cA}$. We set ${\rm H}_\mathcal{AC}^2({\cA},\Co) = {\rm Z}^2_{\mathcal{AC}}\left( {\cA},\Co\right) \big/{\rm B^2}({\cA},\Co)$ and it follows that ${\cA}_\theta$ is anticommutative if and only if $\theta\in{\rm Z}^2_{\mathcal{AC}}\left( {\cA},\Co\right)$. We also define $\Delta^a_{ij}=\Delta_{ij}-\Delta_{ji}$, for $1\leq i, j\leq n$, so that $\left\{\Delta^a_{ij}\mid 1\leq i< j\leq n\right\}$ is a basis of ${\rm Z}_\mathcal{AC}^2({\cA},\Co)$.

\begin{lemma}\label{constructing-T_n+1}
Let $n\geq 5$ and $\cA\in\nil^\gamma_n$ be anticommutative such that  $\ann\cA=\la e_{n}\ra$ and $e_i e_{i+1}=e_{i+2}$, for all $1\le i\le n-2$. Suppose that $\aut(\cA)=\{I+xE_{n1}+yE_{n2}\mid x, y\in\Co\}$ in the basis $(e_i)_{i=1}^n$. Then there is a parametric family of pairwise non-isomorphic, anticommutative $1$-dimensional central extensions $\cB$ of $\cA$ with basis $(e_i)_{i=1}^{n+1}$, extending the basis of $\cA$, such that $\ann\cB=\la e_{n+1}\ra$, $c_{1,n-1}^{n+1}=0=c_{2,n-1}^{n+1}$, $e_i e_{i+1}=e_{i+2}$ for all $1\leq i\leq n-1$, and the structure constants $c_{ij}^{n+1}$ of $\cB$ in this basis are arbitrary independent complex parameters for all $1\leq i\leq j-2\leq n-2$, $(i,j)\ne (1,n-1), (2,n-1)$. 
\end{lemma}
\begin{proof}
The proof is similar to that of Lemma~\ref{constructing-R_n+1}, but there are a few differences, also related to the fact that the automorphism group of $\cA$ is larger. Thus, we will just highlight the differences. We have 
\begin{align*}
{\rm B}^2(\cA)=\bigoplus_{m=2}^{n-1}\La\Delta^{a}_{m-1,m}+\sum_{i< j\leq m, (i, j)\neq (m-1,m)}c_{ij}^{m+1}\Delta^{a}_{ij}\Ra,
\end{align*}
so ${\rm H}^2_{\mathcal{AC}}(\cA)=\la \left[\Delta_{ij}^a\right] \mid 1\leq i\leq j-2\leq n-2\ra\oplus\la \left[\Delta_{n-1,n}^a\right]\ra$.

Let $\phi=I+xE_{n1}+yE_{n2}\in\aut(\cA)$ and $\0=\sum_{1\leq i\leq j-2\leq n-2}\af_{ij}\left[\Delta_{ij}^a\right]+\af_{n-1,n}\left[\Delta_{n-1,n}^a\right]\in {\rm H}^2_{\mathcal{AC}}(\cA)$. Consider the corresponding matrix $A=\sum_{1\leq i\leq j-2\leq n-2}\af_{ij}(E_{ij}-E_{ji})+\af_{n-1,n}(E_{n-1,n}-E_{n, n-1})$. Then, computing $\phi^TA\phi$, we find that $\phi(\0)=\sum_{1\leq i\leq j-2\leq n-2}\af^*_{ij}\left[\Delta_{ij}^a\right]+\af^*_{n-1,n}\left[\Delta_{n-1,n}^a\right]$, where
\begin{align*}
\alpha_{1j}^* &=\alpha_{1j}-x\alpha_{jn},\ \mbox{for $3\leq j\leq n-1$}, \\
\alpha_{2j}^* &=\alpha_{2j}-y\alpha_{jn},\ \mbox{for $4\leq j\leq n-1$,\quad and}\\
\alpha_{ij}^* &=\alpha_{ij}, \ \mbox{otherwise.}
\end{align*}
For $\alpha_{n-1, n}\neq 0$, we can take $x=\alpha_{1,n-1}\alpha_{n-1, n}^{-1}$ and $y=\alpha_{2,n-1}\alpha_{n-1, n}^{-1}$, which gives the family of representatives of distinct orbits
\begin{equation*}
 \La\sum_{\substack{1\leq i\leq j-2\leq n-2\\ (i,j)\ne (1,n-1), (2,n-1)}}c_{ij}^{n+1}\left[\Delta_{ij}^a\right]+\left[\Delta_{n-1,n}^a\right]\Ra_{c_{ij}^{n+1}\in\Co},
\end{equation*}
as claimed.
\end{proof}

We need yet another restriction on the structure constants of the algebras in $\hat\cT_7$ to ensure that different parameter choices give different isomorphism classes.

\begin{lemma}\label{L:anti:7:isoparameters}
Let $\cA\in\hat\cT_7$ with $c_{15}^{7}=0=c_{25}^{7}$. Then $\cA$ is isomorphic to a unique algebra $\cA'\in\hat\cT_7$ with ${d}_{15}^{7}=0={d}_{25}^{7}$ and ${d}_{46}^7=1$, where the ${d}_{ij}^k$ are the structure constants of $\cA'$.
\end{lemma}
\begin{proof}
Notice first that the hypotheses on $\cA$ imply that there are exactly $9$ parameters of freedom, namely: $c_{13}^{5}, c_{13}^{7}, c_{14}^{7}, c_{16}^{7}, c_{24}^{7}, c_{26}^{7}, c_{35}^{7}, c_{36}^{7}, c_{46}^{7}$, with $c_{13}^{5}c_{46}^{7}\neq 0$. Let $x=c_{46}^{7}$ and consider the new basis
\begin{align*}
E_1=x e_1,\quad E_2=e_2,\quad E_3=x e_3,\quad E_4=x e_4, \quad E_5=x^2 e_5,\quad E_6=x^3 e_6, \quad E_7=x^5 e_7. 
\end{align*}
Let ${d}_{ij}^k$ be the structure constants of $\cA$ relative to this basis. Then, it is straightforward to see that these structure constants satisfy all of the conditions determined by $\hat\cT_7$, along with ${d}_{15}^{7}=0={d}_{25}^{7}$. For example, 
\begin{align*}
E_3 E_5= x^3 e_3 e_5=x^3(c_{35}^6 e_6+c_{35}^7 e_7)=x^3e_6+x^3 c_{35}^7 e_7=E_6+c_{35}^7x^{-2}E_7,
\end{align*}
so ${d}_{35}^{6}=1$. Moreover, 
\begin{align*}
E_4 E_6= x^4 e_4 e_6=x^4 c_{46}^7e_7=x^5 e_7=E_7,
\end{align*}
and thus ${d}_{46}^7=1$. This proves the existence part of the statement. 

For the uniqueness, with $\cA'$ as in the statement, suppose also that ${c}_{46}^7=1$ and $\phi:\cA\longrightarrow \cA'$ is an isomorphism. Then, $\phi$ induces an isomorphism $\overline\phi$ on the quotient algebras by their respective annihilators. By Example~\ref{ex:anticomm:tab}, it follows that $c_{13}^5=d_{13}^5$ and $\overline\phi$ is thus an automorphism. We can thence apply Lemma~\ref{L:T:autT6} to get the matrix of $\overline\phi$ and then lift it to $\phi$. From this point it is a simple matter to use the multiplication tables of $\cA$ and $\cA'$ and the matrix of $\phi$ to deduce that the respective free parameters in $\cA$ and $\cA'$ are equal.
\end{proof}

At last we define the family $\cT_n$, which will be shown to be generic in the variety of $n$-dimensional anticommutative algebras.

\begin{definition}\label{T_n-defn}
Let $n\ge 6$ (in case $n=6$, the condition $c_{46}^7=1$ is to be ignored). Denote by $\cT_n$ the family of those algebras in $\hat\cT_n$ such that ${c}_{46}^7=1$ and $c_{1i}^{i+2}=0=c_{2i}^{i+2}$, for all $i\geq 4$. In other words, $\cT_n$ is the family of
$n$-dimensional complex anticommutative algebras whose structure constants $(c_{ij}^k)_{i,j,k}$ relative to the basis $(e_i)_{i=1}^{n}$ satisfy~\eqref{g_ij=0-for-k<=max(ij)} and such that:
\begin{itemize}
\item $e_i e_{i+1}=e_{i+2}$ for all $1\le i\le n-2$;
\item $c_{1i}^{i+2}=0=c_{2i}^{i+2}$, for all $4\leq i\leq n-2$;
\item $c_{13}^{4}=c_{14}^{5}=c_{24}^5=c_{15}^{6}=c_{25}^6=c_{13}^6=0$;
\item $c_{13}^5 \neq 0$;
\item $c_{35}^6=c_{46}^7=1$.
\end{itemize}
The remaining structure constants $c_{ij}^{k}$ are arbitrary, subject only to the anticommutativity constraint.
\end{definition}

\begin{proposition}
Let $n\geq 6$. Then $\dim\left(\overline{\bigcup_{T\in\cT_n}O(T)}\right)=\frac{(n-2)(n^2+2n+3)}{6}$. 
\end{proposition}

\begin{proof}
In case $n=6$, the statement follows from Example~\ref{ex:anticomm:tab} and~\cite[Thm.\ 2]{kkl19}. So assume that $n\geq 7$. The remainder of the proof is again an adaptation of the proof of Proposition~\ref{P:dim:nilp:Rn}.

By Proposition~\ref{P:T:aut}, for every $\cA\in\cT_n$ with $n\geq 7$, we have $\dim\aut \cA=2$, so $\dim O(\cA)=n^2-2$. Moreover, by Lemma~\ref{L:anti:7:isoparameters}, different choices of structure constants in $\cT_7$ give rise to distinct isomorphism classes. Thus, as seen in the proof of Lemma~\ref{L:anti:7:isoparameters}, the isomorphism classes in $\cT_7$ form an $8$-parameter family and the isomorphism classes in $\cT_n$ are obtained by iterated $1$-dimensional central extensions of this family, as shown in Lemma~\ref{constructing-T_n+1}.

Let $r_n$ be the number of independent parameters of the family $\cT_n$. We have $r_7=8$ and $r_{n+1}=r_n+{n-1\choose 2}-2$, for all $n\ge 7$. Therefore, $r_n=\frac{(n-1)(n+1)(n-6)}{6}$. Thus, $\dim\left(\overline{\bigcup_{T\in\cT_n}O(T)}\right)=n^2-2+r_n=\frac{(n-2)(n^2+2n+3)}{6}$.
\end{proof}

Finally, we show that the family $\cT_n$ is generic and determine the dimension of the variety of complex $n$-dimensional anticommutative nilpotent algebras.

\begin{theorem}\label{Tn:generic}
For any $n\ge 6$, the family $\cT_n$ is generic in the variety of $n$-dimensional anticommutative nilpotent algebras. In particular, that variety has dimension $\frac{(n-2)(n^2+2n+3)}{6}$.
\end{theorem}
\begin{proof}
The proof is similar to that of Theorem~\ref{T:Rn:generic}.

For arbitrary $N\in\nil_n^\gamma$ we will prove by induction on $n$ that there is a parametric basis $E_i(t)=\sum_{j=i}^n a_{ji}(t)e_j$, with $a_{ji}(t)\in\Co(t)$, $1\le i\leq j\le n$, and a choice of structure constants $c_{ij}^k(t)\in\Co(t)$, satisfying the conditions of Definition~\ref{T_n-defn} and giving a degeneration of $N$ from $\cT_n$.

The base case $n=6$ has already been proved in Example~\ref{ex:anticomm:tab} and \cite{kkl19}. Denote by $\gm_{ij}^k$ the structure constants of $N$ in $(e_i)_{i=1}^{n+1}$. For the inductive step from $n$ to $n+1$, as in the proof of Theorem~\ref{T:Rn:generic}, it suffices to establish the convergence~\eqref{system-deg-k-is-n+1} by the appropriate choice of $c_{ij}^{n+1}(t)$ and $a_{n+1,i}(t)$. When $n+1=7$, we replace $c_{46}^7=1$ by the more general condition $c_{46}^7\ne 0$, which is permitted in view of Lemma~\ref{L:anti:7:isoparameters}. We may also redefine $a_{n1}(t)$ and $a_{n2}(t)$, as no degeneration from $\cT_n$ depends on these coefficients.

We will proceed in $n-1$ steps, from $k=1$ to $k=n-1$. At the end of {Step $\mathbf{k}$} we will have defined $c_{p,q}^{n+1}(t)$ for all $q>p\geq n-k$ and $a_{n+1,r}(t)$ for all $r\geq n-k+2$. We will also have obtained~\eqref{system-deg-k-is-n+1} for all $n\geq j>i\geq n-k$.
	
	\medskip
	
	\textbf{Step $\mathbf 1$.} Since we must have $c_{n-1,n}^{n+1}(t)=1$, it remains to define $a_{n+1,n+1}(t)$. The left-hand side of~\eqref{system-deg-k-is-n+1} for $i=n-1$ and $j=n$ becomes $\lim_{t\to 0}\left(a_{n-1,n-1}(t)a_{nn}(t)a_{n+1,n+1}(t)^{-1}\right)$, so we set
	\begin{align*}
		a_{n+1,n+1}(t):=
		\begin{cases}
			(\gm_{n-1,n}^{n+1})^{-1}a_{n-1,n-1}(t)a_{nn}(t), & \text{if $\gm_{n-1,n}^{n+1}\ne 0$,}\\
			t^{-1}a_{n-1,n-1}(t)a_{nn}(t), & \text{if $\gm_{n-1,n}^{n+1}=0$.}
		\end{cases}
	\end{align*}
By definition, \eqref{system-deg-k-is-n+1} holds for $i=n-1$ and $j=n$. Notice also that $a_{n1}(t),a_{n2}(t)$ do not occur in the formula above.

\medskip

\textbf{Step $\mathbf{k}$.} Let $2\leq k< n-2$ and assume that {Step $\mathbf{k-1}$} has been successfully completed and that $a_{n1}(t),a_{n2}(t)$ have not been used to define any new coefficients.

Suppose first that $n\geq j>i=n-k$ and $j\ne i+1$. We will define $c_{ij}^{n+1}(t)$ and establish~\eqref{system-deg-k-is-n+1} in this case. The coefficient of $c_{ij}^{n+1}(t)$ on the left-hand side of~\eqref{system-deg-k-is-n+1} equals $a_{n+1,n+1}(t)^{-1}a_{ii}(t)a_{jj}(t)$ which has already been defined and is non-zero. We thus put
	\begin{align}
		c_{ij}^{n+1}(t):=
		\begin{cases}
		    \frac{a_{n+1,n+1}(t)}{a_{ii}(t)a_{jj}(t)}\left(\gm_{ij}^{n+1}-\sum\limits_{r=2}^{n+1}a'_{n+1,r}(t)\sideset{}{'}\sum\limits_{p=i,q=j}^{r-1}c_{pq}^r(t)a_{pi}(t)a_{qj}(t)\right), & \text{if this is non-zero},\label{c_ij^(n+1)=ac}\\
		    \frac{t a_{n+1,n+1}(t)}{a_{ii}(t)a_{jj}(t)}, & \text{otherwise},
		\end{cases}
	\end{align}
where the primed sum is over all $(p,q)$ such that $(p,q,r)\ne(i,j,n+1)$. Note that on the right-hand side of~\eqref{c_ij^(n+1)=ac} we must have $n-k+1<j\leq r-1$, so $r\geq n-k+3$. Thence, by {Step $\mathbf{k-1}$} and Lemma~\ref{a'_ii-in-terms-of-a_ij}, all the terms of the form $a'_{n+1,r}(t)$ on the right-hand side of~\eqref{c_ij^(n+1)=ac} have already been defined. The same holds for all remaining terms except those of the form $c_{iq}^{n+1}(t)$ with $q> j$. Thus, \eqref{c_ij^(n+1)=ac} is a recurrence formula which defines $c_{ij}^{n+1}(t)$ in terms of $c_{iq}^{n+1}(t)$ with $q>j$. So, starting recursively with $c_{in}^{n+1}(t)$, we can define all of the terms $c_{n-k,q}^{n+1}(t)$, with $q>n-k+1$ and by doing so we force the convergence~\eqref{system-deg-k-is-n+1} for all $j>n-k+1$ and $i=n-k$. This will work because we are assuming that $k<n-2$ so $(n-k,q)\neq (1,n-1)$, $(2,n-1)$. Moreover, also by that assumption on $k$, the coefficients $a_{n1}(t),a_{n2}(t)$ have not been used in~\eqref{c_ij^(n+1)=ac} to define $c_{ij}^{n+1}(t)$, as $j>i = n-k\geq 3$. Hence, given that $c_{n-k,n-k+1}^{n+1}(t)=0$ ($k>1$), all $c_{p,q}^{n+1}(t)$ with $q>p\geq n-k$ are defined and~\eqref{system-deg-k-is-n+1} holds for all $j>i\geq n-k$ except if $i=n-k$ and $j=n-k+1$.

Next, we will define $a_{n+1,n+2-k}(t)$ so that~\eqref{system-deg-k-is-n+1} holds for $i=n-k$ and $j=n-k+1$. Assume thus that $i=n-k$ and $j=n-k+1$. Using Lemma~\ref{a'_ii-in-terms-of-a_ij} we have
{\tiny\begin{align*}
	\sum_{r=2}^{n+1}\sum_{p=i,q=i+1}^{r-1}a'_{n+1,r}(t)c_{pq}^r(t)a_{pi}(t)a_{q,i+1}(t)&=a'_{n+1,i+2}(t)a_{ii}(t)a_{i+1,i+1}(t)+\sum_{r=i+3}^{n+1}a'_{n+1,r}(t)\sum_{p=i,q=i+1}^{r-1}c_{pq}^r(t)a_{pi}(t)a_{q,i+1}(t)\\
	&=-a_{n+1,n+1}(t)^{-1}a_{i+2,i+2}(t)^{-1}a_{ii}(t)a_{i+1,i+1}(t)a_{n+1,i+2}(t)\\
	&\quad-a_{i+2,i+2}(t)^{-1}a_{ii}(t)a_{i+1,i+1}(t)\sum_{s=i+3}^n a'_{n+1,s}(t)a_{s,i+2}(t)\\
	&\quad+\sum_{r=i+3}^{n+1}a'_{n+1,r}(t)\sum_{p=i,q=i+1}^{r-1}c_{pq}^r(t)a_{pi}(t)a_{q,i+1}(t).
	\end{align*}}
	Hence, we put
	\begin{align}
	a_{n+1,i+2}(t):=&-\frac{a_{n+1,n+1}(t)a_{i+2,i+2}(t)}{a_{ii}(t)a_{i+1,i+1}(t)}\left(\gm_{i,i+1}^{n+1}-\sum_{r=i+3}^{n+1}a'_{n+1,r}(t)\sum_{p=i,q=i+1}^{r-1}c_{pq}^r(t)a_{pi}(t)a_{q,i+1}(t)\right)\notag\\
	&\quad-a_{n+1,n+1}(t)\sum_{s=i+3}^{n}a'_{n+1,s}(t)a_{s,i+2}(t),\label{a_(n+1_n+1)(t)=ac}
	\end{align}
where the right-hand side defines $a_{n+1,n-k+2}(t)$ in terms of $a_{n+1,r}(t)$ with $r\geq n-k+3$ (already defined in the previous steps) and $c_{p,q}^{n+1}(t)$ with $q>p\geq n-k$ (defined above). Also, \eqref{system-deg-k-is-n+1} holds for $i=n-k$ and $j=n-k+1$  and $a_{n1}(t),a_{n2}(t)$ do not occur in the definition~\eqref{a_(n+1_n+1)(t)=ac} above. This step is thus finished.

\medskip

\textbf{Step $\mathbf{n-2}$.} When we reach this step, all $c_{p,q}^{n+1}(t)$ with $q>p\geq 3$ and all $a_{n+1,r}(t)$ with $r\geq 5$ have been defined without using the coefficients $a_{n1}(t), a_{n2}(t)$ and~\eqref{system-deg-k-is-n+1} has been shown to hold for all $j>i\geq 3$. 

Consider first the case $(i,j)=(2,n)$. Then, from the right-hand side of~\eqref{system-deg-k-is-n+1}, we get
\begin{align*}
    \sum_{r=2}^{n+1}a'_{n+1,r}(t)\sum_{p=i,q=j}^{r-1}c_{pq}^r(t)a_{pi}(t)a_{qj}(t)&=a'_{n+1,n+1}(t)a_{nn}(t)\sum_{p=2}^n c_{pn}^{n+1}(t)a_{p2}(t)\\
    &=a_{n+1,n+1}(t)^{-1}a_{nn}(t)\sum_{p=2}^{n-1} c_{pn}^{n+1}(t)a_{p2}(t),
\end{align*}
so we put
\begin{align*}
    c_{2n}^{n+1}(t):=a_{22}(t)^{-1}a_{n+1,n+1}(t)a_{nn}(t)^{-1}\gm_{2n}^{n+1}-a_{22}(t)^{-1}\sum_{p=3}^{n-1} c_{pn}^{n+1}(t)a_{p2}(t),
\end{align*}
in which the right-hand side has already been defined in the previous steps and does not depend on $a_{n1}(t)$ or $a_{n2}(t)$.

Now let $(i,j)=(2,n-1)$. We will redefine $a_{n2}(t)$ at this point. This is necessary because we are bound to having $c^{n+1}_{2,n-1}(t)=0$, so we cannot force~\eqref{system-deg-k-is-n+1} in the case $(i,j)=(2,n-1)$ by choosing $c^{n+1}_{2,n-1}(t)$ as we please. 
We have 
	\begin{align*}
		\sum_{r=2}^{n+1}\sum_{p=i,q=j}^{r-1}a'_{n+1,r}(t)c_{pq}^r(t)a_{pi}(t)a_{qj}(t)&=a_{n+1,n+1}(t)^{-1}a_{n-1,n-1}(t)\sum_{p=3}^{n}c_{p,n-1}^{n+1}(t)a_{p2}(t)\\
		&\quad+a_{n+1,n+1}(t)^{-1}a_{n,n-1}(t)\sum_{p=2}^{n-1}c_{p,n}^{n+1}(t)a_{p2}(t)\\
		&\quad+a'_{n+1,n}(t) a_{n-1,n-1}(t)\sum_{p=2}^{n-1}c_{p,n-1}^{n}(t)a_{p2}(t),
	\end{align*}
	in which $a'_{n+1,n}(t)=-a_{n+1,n+1}^{-1}(t)a_{nn}^{-1}(t)a_{n+1,n}(t)$ has already been defined and the coefficient of $a_{n2}(t)$ is $-a_{n+1,n+1}(t)^{-1}a_{n-1,n-1}(t)\ne 0$. Hence we set
	\begin{align*}
		a_{n2}(t)&:=-\frac{a_{n+1,n+1}(t)\gm_{2,n-1}^{n+1}}{a_{n-1,n-1}(t)}+\sum_{p=3}^{n-1}c_{p,n-1}^{n+1}(t)a_{p2}(t)+\frac{a_{n,n-1}(t)}{a_{n-1,n-1}(t)}\sum_{p=2}^{n-1}c_{p,n}^{n+1}(t)a_{p2}(t)\\
		&\quad-\frac{a_{n+1,n}(t)}{a_{nn}(t)}\sum_{p=2}^{n-1}c_{p,n-1}^{n}(t)a_{p2}(t),
	\end{align*}
the right-hand side of which has already been defined and does not involve $a_{n1}(t)$ or $a_{n2}(t)$. 

Now we can proceed as in the previous (generic) step with $k=n-2$ defining $c_{2q}^{n+1}(t)$ for $n-2\geq q\geq 3$ and finally $a_{n+1,4}(t)$, ensuring that~\eqref{system-deg-k-is-n+1} holds in the remaining cases. 

\medskip

\textbf{Step $\mathbf{n-1}$.} This step is totally analogous to the previous one. We define $c_{1n}^{n+1}(t)$, then redefine $a_{n1}(t)$ and after that find $c_{1q}^{n+1}(t)$ for $n-2\geq q\geq 2$ and finally $a_{n+1,3}(t)$, ensuring that~\eqref{system-deg-k-is-n+1} holds in the remaining cases. The coefficients $a_{n+1,1}(t)$ and $a_{n+1,2}(t)$ are unrestrained and can be chosen arbitrarily (which agrees with our previous observations).
\end{proof}

\section{Corollaries, an open question and a conjecture}\label{S:final}

\subsection{Nilpotency index}
Recall from Subsection~\ref{SS:nilpalgs} that the nilpotency index of a nilpotent algebra $\cA$ is the smallest positive $k$ such that $\cA^k=0$. The nilpotency index of an $n$-dimensional nilpotent algebra is not greater than $2^{n-1}+1$ but, in general, this upper bound can be attained. On the other hand, for $n$-dimensional nilpotent Lie algebras, it is known that the upper bound for the nilpotency index is $n-1$, while for $n$-dimensional nilpotent associative, Leibniz and Zinbiel algebras, the upper bound on the nilpotency index is $n$.

Thanks to our Theorem~\ref{Tn:generic}, we know that each $n$-dimensional nilpotent anticommutative algebra is a degeneration from the generic family $\cT_n$. This implies the following result.

\begin{corollary}\label{coranti}
The nilpotency index of  an $n$-dimensional nilpotent anticommutative algebra is at most $F_n+1$, where $F_n$ is the $n^{\text{th}}$ Fibonacci number. This bound is sharp and it is attained by the algebras from the generic family 
$\cT_n$ given in Theorem~\ref{Tn:generic} (see Definition~\ref{T_n-defn}).
\end{corollary}

\subsection{Length of algebras}
Let $\cA$ be a finite-dimensional algebra and $S$ be a finite subset of $\cA$. We define the length function of $S$ as follows (see~\cite{guterman}). 
Any product (with any choice of bracketing) of a finite number of elements of $S$ is a word in $S$, the number of letters (i.e.\ elements of $S$) in the product being its length. For $i\geq 1$, the set of all words in $S$ having length less than or equal to $i$ is denoted by $S^i$. Then set $\mathcal L_i(S) = \langle  S^i \rangle$, the linear span of $S^i$, and $\mathcal L(S)=\bigcup\limits_{i=1}^\infty \mathcal L_i(S)$. 
\begin{definition}\label{sys_len} 
Assume that $S$ is a finite generating set for the finite-dimensional algebra $\cA$. Then the length of $S$ is defined as $l(S)=\min\{i\geq 1\mid \mathcal L_i(S)=\cA\}$. The length of $\cA$ is 
\begin{equation*}
l(\cA)=\sup\limits \{l(S)\mid S\subseteq\cA\ \mbox{finite and\ } \mathcal L(S)=\mathcal{A}\}. 
\end{equation*}
\end{definition}

The length of the associative algebra of matrices of size $3$ was first discussed in~\cite{SpeR60} in the context of the mechanics of isotropic continua. The more general problem for the algebra of matrices of size $n$ was posed in \cite{Paz84} but is still open
(recently, some interesting new results about this problem are given by Shitov~\cite{shitov}). The known upper bounds for the length of the matrix algebra of size $n$ are in general nonlinear in $n$.

For our main corollary, we need the following key lemma.
\begin{lemma}\label{lemmanti}
Let $\cA$ be an $n$-dimensional anticommutative algebra of length $k$. Then there is an $n$-dimensional nilpotent anticommutative algebra with nilpotency index $k+1$.
\end{lemma}

\begin{Proof}
Let $S=\{a_1, \ldots, a_t\}$ be a generating set of $\mathcal{A}$ such that $l(\cA)=k=l(S)$.
Our idea for the construction of an $n$-dimensional nilpotent anticommutative algebra $(\cB, \star)$ with nilpotency index $k+1$ is based on a reduction of the multiplication of $\mathcal{A}$ to a nilpotent case. The reduction of the multiplication of $\mathcal{A}$ is given in the following $k+1$ steps.

\medskip

\textbf{Step $\mathbf 1$.}
We consider an algebra $\cB$ with the same underlying vector space as $\cA$.

\medskip

\textbf{Step $\mathbf 2$.}
Fix a complement $K_2$ for $\mathcal L_1(S)$ in $\mathcal L_2(S)$, so that $\mathcal L_2(S)=K_2\oplus \mathcal L_1(S)$. For each pair $a, b$ of elements from $\mathcal L_1(S)$, the product 
$ab$ can be written as $ab= \ell+\ell^*$, where $\ell \in K_2$ and $\ell^* \in \mathcal L_1(S)$. Then set $a \star b =\ell$.

\medskip

\textbf{Step $\mathbf R$.} Let $3\leq R \leq k$.
Fix a complement $K_R$ for $\mathcal L_{R-1}(S)$ in $\mathcal L_R(S)$, so that $\mathcal L_R(S)=K_R\oplus \mathcal L_{R-1}(S)$. For each pair $a,b$ of elements  $a \in \mathcal L_{R-q}(S) \setminus  \mathcal L_{R-q-1}(S)$ and $b \in \mathcal L_{q}(S) \setminus  \mathcal L_{q-1}(S)$, where $1\leq q<R$, the product $ab$ can be written as 
$ab= \ell+\ell^*$, where $\ell \in K_R$ and $\ell^* \in \mathcal L_{R-1}(S)$. Then set $a \star b =\ell$.

\medskip

\textbf{Step $\mathbf{k+1}$.}
The remaining multiplications are zero.

\medskip

By construction, $(\cB, \star)$ is an $n$-dimensional nilpotent anticommutative algebra of length $k$ and nilpotency index $k+1$.
\end{Proof}

Combining Corollary~\ref{coranti} and Lemma~\ref{lemmanti} gives the following corollary.

\begin{corollary}\label{C:final:length}
The length of an $n$-dimensional  anticommutative algebra is bounded above by $F_n$, the $n^{\text{th}}$ Fibonacci number. 
This bound is sharp and it is reached by the algebras from the generic family $\cT_n$ given in Theorem~\ref{Tn:generic} (see Definition~\ref{T_n-defn}).
\end{corollary}

\subsection{Open question and conjecture}
The study of $n$-ary algebras is an interesting topic which has seen good developments recently.
There are many different generalizations to the $n$-ary case of the commutative and anticommutative properties.
Let us give a more general definition below.

\begin{definition}
Let $\mathcal{N}$ be an $n$-ary algebra with multiplication $[ \cdot, \ldots, \cdot]$.
Given two disjoint subsets $A$ and $C$ of $\{1, \ldots ,n\}$,
we say that $\mathcal{N}$ is an $(A,C)$-commutative $n$-ary algebra if the following holds:
\begin{enumerate}
\item the multiplication $[ x_1, \ldots, x_n]$ is anticommutative on elements indexed by $A$,
\item the multiplication $[ x_1,  \ldots, x_n]$ is commutative on elements indexed by $C$.
\end{enumerate}\end{definition}

The main examples of 
$(A,C)$-commutative $n$-ary algebras
are:
\begin{itemize}
    \item Lie triple systems and  Tortkara triple systems ($( \{1,2\}, \emptyset )$-commutative $3$-ary algebras), 
    \item anti-Jordan triple systems  \cite{hader} and  algebraic $N = 6$ $3$-algebras \cite{cantarini} ($( \{1,3\}, \emptyset )$-commutative $3$-ary algebras),
       \item   algebraic $N = 5$ $3$-algebras \cite{cantarini} ($( \emptyset, \{1,2\} )$-commutative $3$-ary algebras),
    \item Jordan quadruple systems  \cite{bremner}
($( \emptyset, \{1,4\})$-commutative $4$-ary algebras and also $( \emptyset, \{2,3\})$-commutative $4$-ary algebras),
    \item commutative (resp.\ anticommutative) $n$-ary algebras
($(\emptyset, \{1, \ldots, n\} )$-commutative (resp.\ $( \{1, \ldots, n\}, \emptyset )$-commutative) $n$-ary algebras).
\end{itemize}
The special case of $(\{1,\ldots, a\},\{n-c+1, \ldots, n\})$-commutative $n$-ary algebras
we will called $(a,c)$-commutative $n$-ary algebras.

The geometric study of varieties of $n$-ary algebras defined by a family of polynomial  identites has been started in \cite{kv20}.
Hence, we have an obvious open question.

\begin{question}
It is clear that the variety of $k$-dimensional nilpotent 
 $(A,C)$-commutative $n$-ary   algebras is irreducible.
What is the geometric dimension of this variety?
\end{question}

In order to formulate our conjecture concerning a bound on the length of $k$-dimensional $(A,C)$-commutative $n$-ary algebras we need to introduce the $N$-generated Fibonacci numbers.
\begin{definition}
Let $N=p_1^{a_1}\cdots p_r^{a_r}$ be the prime decomposition of $N$, where $p_r$ denotes the $r^{\text{th}}$ prime number.
We can define the $N$-generated Fibonacci number $F_{N}(n)$ recursively as
\begin{center}
$F_{N}(n)=a_{1}F_{N}(n-1)+a_{2}F_{N}(n-2)+\cdots+a_{r}F_{N}(n-r)$,
\end{center}
where $F_N(n)=1$ if $n \leq r$. 
\end{definition}

\begin{conjecture}
Let  $\mathcal{N}$ be a $k$-dimensional $(A,C)$-commutative $n$-ary algebra.
Then the length of $\mathcal{N}$ is at most 
$F_{2^{n-a+1}\cdot 3 \cdots p_a}(k)$,
where $a=|A|,$
if $|A|>1$ and $a=1,$ if $|A|=0.$
\end{conjecture}

\begin{remark}
If the conjecture is true, then the bound is sharp and it gives the sharp bound for the nilpotency index of 
$k$-dimensional nilpotent $(a, c)$-commutative $n$-ary   algebras. 
In the case of $k$-dimensional $(a, c)$-commutative $n$-ary   algebras,
it is confirmed by the following $n$-ary algebra $\mathcal N$ with the multiplication given by  \begin{center}
    $[e_{j-a+1}, \ldots, e_{j-1}, e_j, \ldots, e_j ]=e_{j+1},$ \ $ a \leq j \leq k-1.$
\end{center}
For arbitrary  $(A,C)$-commutative $n$-ary algebras, 
an extremal example can be obtained by a similar way using a suitable permutation of the indices in the above multiplication.
\end{remark}


\begin{thebibliography}{}
 
    
 
 

 


\bibitem{ale3}
Alvarez M.A., 
Degenerations of $8$-dimensional $2$-step nilpotent Lie algebras,
Algebras and Representation Theory, 2020, DOI: 10.1007/s10468-020-09987-5
 
 
 
 




\bibitem{bremner}
Bremner M., Madariaga S., 
    Jordan quadruple systems, 
    Journal of Algebra, 412 (2014), 51--86.

 
 

\bibitem{BC99} 
Burde D., Steinhoff C.,
    Classification of orbit closures of $4$--dimensional complex Lie algebras,
    Journal of Algebra, 214 (1999), 2, 729--739.
 
 
   \bibitem{cantarini}  
   Cantarini N., Kac V., 
    Classification of linearly compact simple algebraic $N=6$ $3$-algebras, 
    Transformation Groups,  16 (2011), 3, 649--671.
 
   \bibitem{chouhy}
Chouhy S.,
    On geometric degenerations and Gerstenhaber formal deformations,
    Bulletin of the London Mathematical Society, 51 (2019),  5, 787--797.

   \bibitem{cibils}  Cibils C., 
    $2$-nilpotent and rigid finite-dimensional algebras,
    Journal of the London Mathematical Society (2), 36 (1987), 2, 211--218.

 
 \bibitem{hader}
  Elgendy H., 
    The universal associative envelope of the anti-Jordan triple system of $n \times n$ matrices, 
    Journal of Algebra, 383 (2013), 1--28.
    
\bibitem{fkkv}
Fern\'andez Ouaridi A.,  Kaygorodov I.,  Khrypchenko M., Volkov Yu., 
    Degenerations of nilpotent algebras,
    arXiv:1905.05361

\bibitem{fF68}
Flanigan F.\ J., 
    Algebraic geography: {V}arieties of structure constants, 
    Pacific Journal of Mathematics, 27 (1968), 71--79.
    
   \bibitem{gabriel}
Gabriel P.,
Finite representation type is open,
Proceedings of the International Conference on Representations of Algebras (Carleton Univ., Ottawa, Ont., 1974), pp. 132--155.
 
\bibitem{ger63}
Gerstenhaber M.,
    On the deformation of rings and algebras,
    Annals of Mathematics (2), 79 (1964), 59--103.
	

\bibitem{gorb93} 
Gorbatsevich V., 
    Anticommutative finite-dimensional algebras of the first three levels of complexity, 
    St. Petersburg Mathematical Journal, 5 (1994), 505--521.

 
\bibitem{GAB92}
Goze M., Ancoch\'{e}a-Berm\'{u}dez J.\ M.,
    On the varieties of nilpotent {L}ie algebras of dimension {$7$} and {$8$},
    Journal of Pure and Applied Algebra, 77 (1992), 131--140.



\bibitem{GRH}
Grunewald F.,  O'Halloran J.,
    Varieties of nilpotent Lie algebras of dimension less than six,
    Journal of Algebra, 112 (1988), 2, 315--325.

 \bibitem{guterman}
 Guterman A.,  Kudryavtsev D., 
    Upper bounds for the length of non-associative algebras,  
    Journal of Algebra, 544 (2020), 483--497.

 
 \bibitem{ikp20}
 Ignatyev M.,  Kaygorodov I., Popov Yu., 
  The geometric classification of $2$-step nilpotent algebras   and applications, preprint

 
\bibitem{ikv17}
Ismailov N., Kaygorodov I.,  Volkov Yu.,
    The geometric classification of Leibniz algebras,
    International Journal of Mathematics, 29  (2018), 5, 1850035.


 
 
 
 

\bibitem{kkl19}
Kaygorodov I., Khrypchenko M., Lopes S.,  
    The algebraic and geometric classification of nilpotent anticommutative algebras, 
    Journal of Pure and Applied Algebra,  224  (2020), 8, 106337. 
  
 
\bibitem{kkp19geo}
Kaygorodov I.,  Khrypchenko M.,  Popov Yu.,
    The algebraic and geometric classification of nilpotent terminal algebras,
     Journal of Pure and Applied Algebra,  225 (2021), 6, 106625.




 
  

 

\bibitem{kpv19}
Kaygorodov I., P\'{a}ez-Guill\'{a}n  P., Voronin V.,  
    The algebraic and geometric classification of nilpotent bicommutative algebras,
    Algebras and Representation Theory, 23 (2020), 6, 2331-2347.


 

 

 \bibitem{kv17}
Kaygorodov I., Volkov Yu., 
    Complete classification of algebras of level two,  
    Moscow Mathematical Journal, 19 (2019), 3,  485--521.
 

\bibitem{kv19}
 Kaygorodov I.,  Volkov Yu.,
The variety of $2$-dimensional algebras over an algebraically closed field,
Canadian Journal of Mathematics,  71 (2019),  4, 819--842.



 \bibitem{kv20}
Kaygorodov I.,  Volkov Yu., 
    Degenerations of Filippov algebras,  
    Journal of Mathematical Physics, 61 (2020), 2, 021701, 10 pp. 
 
 


\bibitem{maz79}
Mazzola G.,
    The algebraic and geometric classification of associative algebras of dimension five, 
    Manuscripta Mathematica, 27 (1979), 1, 81--101. 


\bibitem{NR66}
Nijenhuis A., Richardson R.\ W., Jr.,
    Cohomology and deformations in graded {L}ie algebras, 
    Bulletin of the American Mathematical Society, 72 (1966), 1--29. 
 
 
 
\bibitem{Paz84} Paz A.,  
    An application of the Cayley--Hamilton theorem to matrix polynomials in several variables, 
    Linear Multilinear Algebra,   15 (1984), 2,  161--170.

 
\bibitem{shaf}
    Shafarevich I., 
    Deformations of commutative algebras of class $2,$ Leningrad Mathematical Journal, 2 (1991), 6, 1335--1351.


\bibitem{SpeR60} Spencer A.J.M., Rivlin R.S., 
 Further results in the theory of matrix polynomials, 
 Archive for Rational Mechanics and Analysis,   4 (1960), 214--230.

\bibitem{shitov}
Shitov Ya., 
    An improved bound for the lengths of matrix algebras,  
    Algebra and Number Theory, 13 (2019), 6, 1501--1507.


\bibitem{SS:method}
Skjelbred T., Sund T.,
    Sur la classification des algebres de Lie nilpotentes,
    C. R. Acad. Sci. Paris Ser. A-B, 286 (1978), 5,  A241--A242.

 
 
 
 
 
 

















\end{thebibliography}
\end{document}